\pdfoutput=1 
\documentclass[a4paper]{article}
\usepackage{amsthm,amssymb,amsmath,enumerate,graphicx,tikz-cd}

\DeclareMathOperator{\Ker}{Ker}
\DeclareMathOperator{\Ima}{Im}

\newtheorem{THM}{Theorem}[section]
\newtheorem{LEM}[THM]{Lemma}
\newtheorem{COR}[THM]{Corollary}

\newtheorem{PROP}[THM]{Proposition}

\newtheorem{EX}[THM]{Example}
\newtheorem{PROB}[THM]{Problem}

\def\specrelabove#1#2{\mathrel{\mathop{\kern0pt #1}\limits^{#2}}}
\def\specrelbelow#1#2{\mathrel{\mathop{\kern0pt #1}\limits_{#2}}}
\def\restricts{\!\restriction\!}
\newcommand\B{\mathcal B}
\def\C{\mathcal C}
\def\D{\mathcal D}
\def\E{\mathcal E}

\newcommand\FF{\mathbb F}

\newcommand\U{\mathcal U}
\newcommand\V{\mathcal V}
\newcommand\R{\mathbb R}
\newcommand\Z{\mathbb Z}

\def\lowfwd #1#2#3{{\mathop{\kern0pt #1}\limits^{\kern#2pt\raise.#3ex \vbox to 0pt{\hbox{$\scriptscriptstyle\rightarrow$}\vss}}}}

\def\lowbkwd #1#2#3{{\mathop{\kern0pt #1}\limits^{\kern#2pt\raise.#3ex
\vbox to 0pt{\hbox{$\scriptscriptstyle\leftarrow$}\vss}}}}

\def\vSdash{{\mathop{\kern0pt S\lower-1pt\hbox{${}
     \scriptstyle'$}}\limits^{\kern2pt\raise.1ex
     \vbox to 0pt{\hbox{$\scriptscriptstyle\rightarrow$}\vss}}}}

\def\vsdash{{\mathop{\kern0pt s\lower.5pt\hbox{${}
     \scriptstyle'$}}\limits^{\kern0pt\raise.02ex
     \vbox to 0pt{\hbox{$\scriptscriptstyle\rightarrow$}\vss}}}}
\def\svdash{{\mathop{\kern0pt s\lower.5pt\hbox{${}
     \scriptstyle'$}}\limits^{\kern0pt\raise.02ex
     \vbox to 0pt{\hbox{$\scriptscriptstyle\leftarrow$}\vss}}}}

\def\vxdash{{\mathop{\kern0pt x\lower.5pt\hbox{${}
     \scriptstyle'$}}\limits^{\kern0pt\raise.02ex
     \vbox to 0pt{\hbox{$\scriptscriptstyle\rightarrow$}\vss}}}}
\def\xvdash{{\mathop{\kern0pt x\lower.5pt\hbox{${}
     \scriptstyle'$}}\limits^{\kern0pt\raise.02ex
     \vbox to 0pt{\hbox{$\scriptscriptstyle\leftarrow$}\vss}}}}

\def\es{\emptyset}
\def\sub{\subseteq}
\def\supe{\supseteq}
\def\sm{\smallsetminus}

\def\partialone{{\partial}}
\def\deltanought{{\delta}}
\def\gammanought{{\gamma}}
\def\gammaone{{\gamma}}

\newcommand\COMMENT[1]{}
\def\?#1{\vadjust{\vbox to 0pt{\vss\vskip-8pt\leftline{%
     \llap{\hbox{\vbox{\pretolerance=-1
     \doublehyphendemerits=0\finalhyphendemerits=0
     \hsize30truemm\tolerance=10000\small
     \lineskip=0pt\lineskiplimit=0pt
     \rightskip=0pt plus16truemm\baselineskip8pt\noindent
     \hskip0pt        
     #1\endgraf}\hskip7truemm}}}\vss}}}
\def\?#1{}

\def\noproof{{\unskip\nobreak\hfill\penalty50\hskip2em\hbox{}\nobreak\hfill%
       $\square$\parfillskip=0pt\finalhyphendemerits=0\par}\goodbreak}
\def\endproof{\noproof\bigskip}
\def\looseproof#1{\bigbreak\noindent {{\bf #1.}}}

\title{Homological aspects of oriented hypergraphs}%
   \COMMENT{}
 \author{Reinhard Diestel}
 \date{\today}

\begin{document}
\abovedisplayshortskip=-3pt plus3pt
\belowdisplayshortskip=6pt

\maketitle

\begin{abstract}\noindent
  We define an algebraic setup of homology for hypergraphs, which defaults to simplicial homology in the case of graphs, and study its basic properties. As part of our study we define algebraic spanning trees of hypergraphs, along with fundamental cuts and cycles playing their usual roles.  
   \end{abstract}

\section{Introduction}

The homology of finite graphs, better known in graph theory as the study of the cycle and cut spaces of a graph and their interaction, is both well known and still of interest as a useful tool in sometimes unexpected contexts.%
   \COMMENT{}
   Because of the added structure offered by graphs in terms of spanning trees,%
   \COMMENT{}
   there is more to it than what topologists might consider as the 1-dimensional case of simplicial homology. This is the case particularly for integer rather than real coefficients; see Biggs~\cite{Biggs, BiggsCriticalGroup} for an overview of some hidden depths and surprising connections ranging from chip-firing games to algebraic geometry.

$\!$By contrast, it seems that the natural notion for the homology of hypergraphs which extends that for graphs has never been introduced, let alone studied.%
   \footnote{Oriented hypergraphs have been investigated for their algebraic properties from a spectral point of view over the reals; see e.g.~\cite{JostMulasAIM2019}. Unoriented $k$-uniform hypergraphs on a set~$V\!$ have been studied in homological terms by viewing them as $(k-1)$-chains of the simplicial complex~$2^V\!$ over~$\FF_2$;%
   \COMMENT{}
   see e.g.~\cite{ChungGrahamHypergraphCohomology,CsakanyKahn1999}. Analogues of spanning trees in higher-dimensional complexes have been studied (in particular: counted, so as to generalise Cayley's formula or the matrix-tree theorem) by various authors following~\cite{KalaiSpanningTrees}.%
   \COMMENT{}
   I have been told, but not verified, that much of what is proved in this paper using just elementary linear algebra can be deduced from more general results on the homology of cell complexes developed in~\cite{DKMCellComplexes2015}.}
   The reason may be that, unlike graphs, oriented hypergraphs (which are needed to set up any homology not just over~$\FF_2$) are not just the usual unoriented hypergraphs with an `orientation' added. Or it may be that hypergraphs are not examples of simplicial complexes and thereby naturally endowed with a homology. But there is a boundary operator for oriented hypergraphs that is just as natural as that for oriented graphs, and which defaults to the latter when the hypergraph is a graph. The resulting homology for hypergraphs~-- just in dimensions 0 and~1, as for graphs~-- extends the cycle/cut space theory for graphs in many ways, and differs from it in other ways. It is the aim of this paper to introduce this natural homology and establish its basic properties.

Although our focus will be on the fundamentals of hypergraph homology rather than on applications, it was through a particular application that I became interested in this topic. I was studying duality aspects of set partitions and their tangles~\cite{ASSduality, Focus} in order to describe the purchasing\vadjust{\penalty-200} behaviour of customers in an imaginary online shop~\cite{TangleBook, ASSduality}. This has duality aspects in that one can look at which items a typical customer buys, but also at which customers buy a typical item. It turned out that this duality was a case of the duality between homology and cohomology of such partition systems when suitably set up, which in turn is a special case%
   \COMMENT{}
   of the natural hypergraph homology to be introduced here.\looseness=-1

In addition to setting up the basics of hypergraph homology and asking such questions as how their homology and cohomology groups relate,%
   \COMMENT{}
   we shall focus on how much of the familiar cycle and cut space theory of graphs carries over to hypergraphs. This will lead us to a definition of `algebraic spanning trees' of hypergraphs, complete with fundamental cuts and cycles. We shall prove that all hypergraphs have such spanning trees over the reals, but not necessarily over the integers. Over the reals, such `spanning trees' exist even in  general vector spaces, which sheds a perhaps unexpected graph-theoretic light on the otherwise familiar duality between linear maps.

Several open problems remain. Some of these concern hypergraph homology in general, but most take the form of asking which structural properties of hypergraphs might imply desirable homological properties not shared by all hypergraphs.

There are few prerequisites for reading this paper. Although most of our proofs will be algebraic in nature, I have made a point of keeping them completely elementary so that their analogy to the graph case, wherever it exists or fails, is most transparent. It is be possible to translate the (very basic) theory developed here into matrix language and prove the resulting assertions there, using algebraic tools such as the Smith normal form~\cite{BiggsCriticalGroup,StanleySmith} to make some of the proofs (though not all) quite short. I~believe that this would not be in the interest of the intended readership, which is graph theorists familiar with the cycle/cut space theory for graphs and looking for analogues describing hypergraphs. As the paper is written, it assumes familiarity only with the basic notions of homological algebra, such as chain complexes and exact sequences, and some basic module theory. A~key tool will be the elementary divisor theorem for free $\Z$-modules, which is recalled and applied in an easily verifiable way.\looseness=-1

\section{Basic terminology}

An {\em oriented hypergraph\/}~$H$ in this paper is a pair $(V,E)$ of a finite set~$V\!$ and a set~$E$ of ordered pairs~$(A,B)$ of disjoint subsets of~$V\!$.%
   \COMMENT{}
   The elements of~$V\!$ are the {\em vertices\/} of~$H$; the elements of~$E$ are its (oriented) {\em edges\/}. The elements of~$A$ are the {\em initial vertices\/} of the edge $e=(A,B)$; the elements of~$B$ are its {\em terminal vertices\/}. This definition first appeared in~\cite{DirectedHypergraphsAndApps} and appears to be standard now.%
   \COMMENT{}

As with oriented graphs~\cite{DiestelBook16} we shall assume that $E$ is asymmetric, in that it does not contain both an edge $e=(A,B)$ and its {\em inverse\/} $e^* = (B,A)$. The reason for this is not that pairs of inverse edges will not play a role in our context. On the contrary, they are so ubiquitous that we need to keep track of them: when we pick two arbitrary edges $e,f\in E$ we want to be sure that they are not accidentally inverse to each other, while if we do wish to consider the inverse of an edge~$e$ we want to refer to it in a way that relates it to~$e$, e.g.\ as~$e^*$.

In fact, in our algebraic context it will rarely be necessary to refer inverse edges by name. The edges~$e$ of~$H$ will be the generators of its chain group~$C_1$ and thus have a group inverse there,~$-e$, which is usually all we need to refer to.\looseness=-1 %
  \COMMENT{}

\section{Chains and cochains}\label{sec:chains}

Given a hypergraph $H = (V,E)$, let $C_0$, $C_1$ and $C_2$ denote the free abelian groups, or $\Z$-modules, with bases $V\!$, $E$, and~$\es$, respectively. We write the elements of~$C_0$, the 0-{\em chains\/} of our hypergraph, as sums $\sum n_i v_i$ of elements~$v_i$ of~$V\!$ with integer coefficients $n_i\in\Z$. The elements of~$C_1$, its 1-{\em chains\/}, are the sums $\sum n_i e_i$ of edges~$e_i\in E$ with integer coefficients. The set $C_2 = \{0\}$ of `2-chains' consists only of the empty sum.\looseness=-1

As {\em boundary homomorphisms\/} we take the map $\partial_2\colon 0\mapsto 0$ from $C_2$ to~$C_1$ and the homomorphism $\partial_1\colon C_1\to C_0$ that sends every $e=(A,B)\in E$ to the 0-chain $\sum_{v\in B} v - \sum_{v\in A} v$%
  \COMMENT{}
  (and extends linearly to all of~$C_1$). An arbitrary~$v\in V\!$ thus has a coefficient in~$\partial_1(A,B)$ of~1 if $v\in B$, of $-1$ if $v\in A$, and of~0 otherwise. Informally, the boundary of an edge consists of the vertices it points to minus the vertices it points away from (Figure~\ref{bdry}). When $H$ is a graph, this coincides with the usual definition for simplicial complexes.

\begin{figure}[ht]
 \center
   \includegraphics[scale=1]{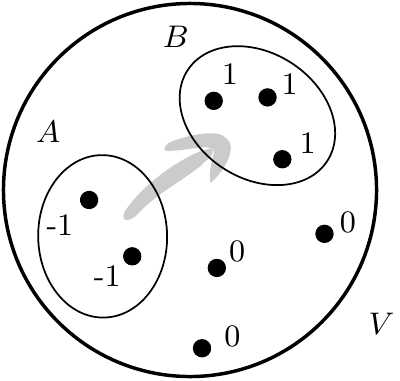}
\caption{Boundary coefficients for an edge~$e=(A,B)$}
\label{bdry}
\end{figure}

If desired, the boundary homomorphism~$\partial_1$ can be described in the usual way by an $n\times m$ matrix $B$ with entries $1,0,-1$,%
   \COMMENT{}
   where $n=|V|$ and $m=|E|$. For a 1-chain $x = \sum_{j=1}^m \alpha_j e_j$ we then have $\partial_1 x = \sum_{i=1}^n \beta_i v_i$ for $Ba=b$, where $a = (\alpha_1,\dots,\alpha_m)$ and $b = (\beta_1,\dots,\beta_n)$.%
  \COMMENT{}

For $n=0,1,2$, the elements of $C^n = {\rm Hom}(C_n,\Z)$ are the $n$-cochains of our hypergraph. We usually define these homomorphisms explictly only on the singleton chains in~$C_n$, i.e., on individual elements of~$V\!$ or of~$E$ when $n=0$ or $n=1$, and extend these maps linearly to all of~$C_n$. We write $C^2 = \{0\}$, where 0~denotes the unique homomorphism $0\mapsto 0$ from~$C_2$ to~$\Z$.

The {\em coboundary homomorphisms\/} for $n=0,1$ are the maps $\delta^n\colon  C^n\to C^{n+1}$ that send an $n$-cochain $\varphi\colon C_n\to\Z$ to the $(n+1)$-cochain $(\varphi\circ\partial_{n+1})\colon C_{n+1}\to\Z$. For example, $\delta^0$~sends $\varphi\in C^0$ to the homomorphism $\psi\colon C_1\to\Z$ that maps every $x\in C_1$ to the image of~$\partial_1 x$ under~$\varphi$. If $\partial_1$ is described by a matrix~$B$ as earlier, then $\delta^0$ is described by its transpose~$B^\top$.%
   \COMMENT{}
   We write~$\varphi_v$ for the unique element of~$C^0$ that sends $v\in V\!$ to~1 and all the other vertices to~0, and $\psi_e$ for the unique element of~$C^1$ that sends $e\in E$ to~1 and all the other edges to~0.

Chains can be turned into cochains simply by interpreting their coefficients as images of the basis element to which they are assigned. Let $\gamma_0$ be the isomorphism $C_0\to C^0$ that maps every $v\in V\!$ to~$\varphi_v$, and $\gamma_1$ the isomorphism $C_1\to C^1$ that maps every $e\in E$ to~$\psi_e$. Thus, $\gamma_0$ maps $\sum_{v\in V} n_v v\in C_0$ to the 0-cochain that sends each $v\in V\!$ to~$n_v$ (and extends linearly to all of~$C_0$), and similarly for~$\gamma_1$.%
  \COMMENT{}

\begin{figure}[ht]
 \center\vskip-6pt
\begin{tikzcd}
0
\arrow[r] 
& C_1
\arrow[r, "\partialone"] \arrow[d, "\gammaone"']
& C_0
\arrow[r] \arrow[d, "\gammanought"]
&0 \\
0
& C^1 \arrow[l]
& C^0 \arrow[l, "\deltanought"]
& 0 \arrow[l] 
\end{tikzcd}
\caption{Boundary and coboundary maps}
\label{chain diagram}
\vskip-\medskipamount
\end{figure}

To avoid clutter, we shall drop the indices 0, 1 or~2 of the maps $\partial$, $\delta$ and~$\gamma$ when they can be understood from the context, as in Figure~\ref{chain diagram}.

For $x\in C_1$ we shall use the abbreviations of
$$\varphi_{\partial x} := (\gammanought\circ \partialone)(x)\in C^0\quad\text{and}\quad \psi_{\partial x} := (\deltanought\circ\gamma\circ\partial)(x) = \delta(\varphi_{\partial x}) \in C^1.$$
For $e\in E\sub C_1$ we thus have $\varphi_{\partial e} (v) = 1$ if $e$ points towards~$v$, and $\varphi_{\partial e} (v) = -1$ if $e$ points away from~$v$. This is compatible with our earlier definition of $\varphi_v$, since if $\partial x = v\in V\sub C_0$%
   \COMMENT{}
   then $\varphi_{\partial x}$ coincides with~$\varphi_v$ as defined earlier.%
   \COMMENT{}

\section{The boundary inner product}\label{sec:innerproduct}%
   \COMMENT{}

Given a chain~$x\in C_1$, consider $\psi_{\partial x} = (\deltanought\circ\gammanought\circ\partialone)(x)$. This is a 1-cochain, a homomorphism $C_1\to\Z$. Thus, $\psi_{\partial x}$ maps every chain $y\in C_1$ to some integer $\psi_{\partial x}(y)\in\Z$. Let us denote this integer by
$$\langle x,y \rangle_\partial := \psi_{\partial x}(y)\in\Z.$$
The form $\langle\ ,\ \rangle_\partial$ is easily seen to be bilinear.%
  \COMMENT{}
  It is also symmetric:

\begin{LEM}\label{symmetry}%
   \COMMENT{}
Let $x,y\in C_1$. Write $\alpha_v$ and~$\beta_v$ for the coefficients of all the~$v\in V\!$ in $\partialone x$ and $\partialone y$, respectively, so that $\partialone x = \sum_{v}\alpha_v v$ and $\partialone y = \sum_{v}\beta_{v} v$. Then\smallskip
\begin{displaymath}\textstyle
  \langle x,y\rangle_\partial = \sum_{v\in V} \alpha_v\beta_v = \langle y,x\rangle_\partial.
\end{displaymath}
\end{LEM}

\bigbreak

\proof We have $\langle x,y \rangle_\partial = \psi_{\partial x} (y) = \deltanought(\varphi_{\partial x})(y) = \varphi_{\partial x} (\partialone y) = \gammanought(\partialone x)(\partial y)$. As $\gammanought(v)(v') = \delta_{vv'}$,%
   \footnote{This is the Kronecker symbol: $\delta_{vv'}=1$ if~$v=v'$, and $\delta_{vv'}=0$ if $v\ne v'$.}
 we obtain
\[\textstyle \langle x,y\rangle_\partial = \gammanought(\partialone x)(\partial y) = \gammanought \big(\sum_v \alpha_v v\big)\big(\sum_{v'} \beta_{v'} v'\big)%
   \COMMENT{}
   = \sum _{v,v'} \alpha_v\beta_{v'}\delta_{vv'} = \sum_v \alpha_v\beta_v\,,\]%
  \COMMENT{}
which is symmetric in $x$ and~$y$ by the definition of $\alpha_v$ and~$\beta_v$.
  \endproof

\begin{figure}[ht]
 \center
   \includegraphics[scale=1]{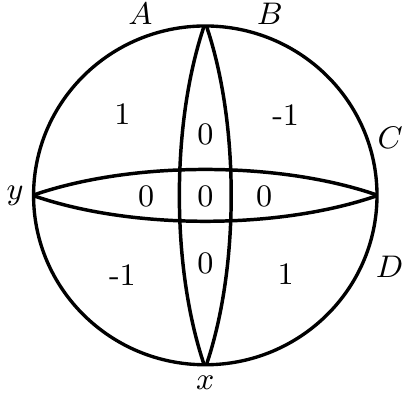}
\caption{The values of~$\alpha_v \beta_v$ for $x = (A,B)$ and $y = (C,D)$}
\label{fig:innerproduct}
\vskip-\medskipamount
\end{figure}

Figure~\ref{fig:innerproduct} illustrates Lemma~\ref{symmetry} when $x$ and~$y$ are single edges. In this case $\langle x,y \rangle_\partial$ measures how similar $x$ and~$y$ are: it counts the vertices on which $x$ and~$y$ agree (by both pointing towards or away from them), deducts the number of vertices on which they disagree,%
   \COMMENT{}
   and ignores the vertices that lie `outside' at least one of the edges $x$ and~$y$.

It is perhaps instructive to compare this with the canonical inner product%
   \COMMENT{}
   $\langle\ ,\ \rangle$ on~$C_1$.%
  \COMMENT{}
  For $x,y\in C_1$ with $x = \sum_i \lambda_i e_i$ and $y = \sum_i \mu_i e_i$, where the $e_i$ run over~$E$, we have $\langle x,y \rangle%
   \COMMENT{}
   = \gammaone(x)(y) =%
  \COMMENT{}
  \sum_i \lambda_i\mu_i$. When $x,y\in E$, then this just indicates whether or not $x$ and~$y$ are identical: $\langle x,y \rangle = \delta_{xy}$. As a similarity measure for two edges $x$ and~$y$, it is rather cruder than~$\langle x,y \rangle_\partial$ which, by Lemma~\ref{symmetry}, is the canonical inner product%
   \COMMENT{}
   of $\partialone x$ and $\partialone y$ rather than of $x$ and~$y$.\looseness=-1

Our $\langle\ ,\ \rangle_\partial$ is not an inner product on~$C_1$: when $x$ or $y$ is an algebraic {\em cycle\/}, an element of $\Ker\partialone$, then clearly $\langle x,y \rangle_\partial = 0$.%
  \COMMENT{}
  However, setting $\langle [x],[y]\rangle_\partial := \langle x,y \rangle_\partial$  yields a well-defined symmetric bilinear form on~$C_1/\Ker\partialone$, which is an inner product there:

\begin{LEM}\label{lem:innerproduct}
On $C_1/\Ker\partialone$, the form $\langle\ ,\ \rangle_\partial$ is an inner product.%
  \COMMENT{}
  \end{LEM}

\proof
We only have to show that $\langle\ ,\ \rangle_\partial$ is positive definite. Given ${x\in C_1}$, let again $\partialone x =:\! \sum_v \alpha_v v$. If $x\notin\Ker\partialone$, the~$\alpha_v$ are not all zero, so ${\langle x,x \rangle_\partial =\! \sum_v \alpha_v^2 > 0}$ by Lemma~\ref{symmetry}.%
   \COMMENT{}
  \endproof


\section{Cycles and cocycles}\label{sec:Cs}%
  \COMMENT{}

Viewed for simplicial complexes,%
  \COMMENT{}
  the diagram in Figure~\ref{cycle-cocyle diagram for graphs} describes some well-known isomorphisms related to the cycles and cocycles%
  \COMMENT{}
  in a connected%
  \COMMENT{}
  graph.%
  \COMMENT{}
  To put into perspective what comes later, let us briefly sketch a few well-known facts related to this diagram, without proofs. The facts summarised here are intended as background only, and will not be needed later except for comparison.

\begin{figure}[h]
 \center
\begin{tikzcd}
C_1\arrow [r, "\pi"]\arrow[d, "\gammaone"']
& C_1/\Ker\partialone
\arrow[r, "\bar\partial", "\simeq"'] \arrow[d, "\sigma"', "\simeq", rightarrow]
& \Ima\partialone\arrow [r, hookrightarrow]
 \arrow[d, "\simeq"', "\tau", rightarrow]
& C_0\arrow[d, "\gammanought"] \\
C^1\arrow [r, hookleftarrow]
& \Ima\deltanought
& C^0/\Ker\deltanought\arrow[l, "\bar\delta"', "\simeq"]
& C^0\arrow[l, "\pi"]
\end{tikzcd}
\caption{Cycles, cocycles, and their quotients in graphs}
\label{cycle-cocyle diagram for graphs}
\end{figure}%
  \COMMENT{}

Given a graph, we usually write $\E = C_1$ for its {\em edge space\/},%
  \footnote{When these spaces are considered with integer rather than $\FF_2$ coefficients, as we do here, they are sometimes referred to as the {\em oriented\/} edge, vertex, cycle and bond space.}
  $\V = C_0$ for its {\em vertex space\/}, $\C = \Ker\partialone$ for its {\em cycle space\/}, and $\B = \gammaone^{-1}(\Ima\deltanought)$ for its {\em bond space}.%
   \COMMENT{}
   The isomorphism $\sigma\colon\E/\C\to\gammaone(\B)$ indicated in Figure~\ref{cycle-cocyle diagram for graphs} is not unique,%
  \COMMENT{}
  but there is a standard way to choose it: pick a spanning tree~$T$, and let $\sigma$ map the cosets $f+\C$ for the oriented edges $f$ of~$T$ to the oriented fundamental cuts~$D_f\in\B$. It is not hard to check that this defines an isomorphism from $\E/\C$ to~$\B$.%
  \COMMENT{}

Once we have chosen an isomorphism~$\sigma$ between $C_1/\Ker\partialone$ and~$\Ima\deltanought$,%
  \COMMENT{}
  the isomorphism $\tau\colon\Ima\partialone\to C^0/\Ker\deltanought$ can be defined simply by composing the other three isomorphisms indicated in Figure~\ref{cycle-cocyle diagram for graphs}. (The two horizontal isomorphisms $\bar\partial$ and~$\bar\delta$ are the canonical ones that $\partialone$ and~$\deltanought$ induce on the quotients $C_1/\Ker\partialone$ and $C_0/\Ker\deltanought$ by the isomorphism theorem.)%
  \COMMENT{}

Defined in this way, $\tau$~has nothing to do with the standard isomorphism $\gammanought\colon C_0\to C^0$ from Figure~\ref{chain diagram},%
  \COMMENT{}
  which depends neither on~$\partialone$ nor on our choice of~$\sigma$. In particular, $\tau$~need not coincide with the map~$\pi\circ\gammanought\colon \Ima\partial_1 \to C^0/\Ker\deltanought$, where $\pi\colon C^0\to C^0/\Ker\deltanought$ is the canonical projection. In fact, it is not hard to construct examples where $\pi\circ\gammanought$ is not even surjective.%
   \COMMENT{}

When we take real coefficients, however, then $\pi\circ\gammanought\restricts\Ima\partial_1$ is another isomorphism. Indeed, it is not hard to show (with any coefficients) that $\Ima\partial_1$ consists of those 0-chains whose coefficients sum to zero,%
  \COMMENT{}
     while the kernel of~$\delta^0$ consists of those 0-cochains that are constant on~$V\!$.%
  \COMMENT{}
  Now, with either integer or real coefficients,%
   \COMMENT{}
   ${0\in C_0}$ is the only 0-chain whose coefficients sum to zero while being identical. Hence $\gammanought$ maps only ${0\in\Ima\partial_1}$ to~$\Ker\delta^0$, and hence remains injective when composed with~$\pi$.%
  \COMMENT{}
  When we take real coefficients and our modules are vector spaces, the existence of the isomorphism~$\tau$, which we deduced earlier from the existence of~$\sigma$, implies that ${\pi\circ\gammanought\colon\Ima\partial_1\to C^0/\Ker\deltanought}$ is surjective too, and hence an isomorphism.%
  \COMMENT{}

\medbreak

Let us now look at the analogous situation for hypergraphs, once more with integer coefficients. We no longer have natural isomorphisms~$\sigma$ now, but $C_1/\Ker\partialone$ still has an isomorphic counterpart inside the image of~$\delta^0$ obtained by composing the two canonical monomorphisms $\bar\partialone\colon C_1/\Ker\partialone\to C_0$ and $\bar\deltanought\colon C^0/\Ker\deltanought \to C^1$ with the isomorphism $\gammanought\colon C_0\to C^0$ from Figure~\ref{chain diagram}.

Indeed, let us write $\U$ for the image of~$\pi\circ\gammanought\circ\partialone$ in~$C^0/\Ker\deltanought$, and $\D$ for the image of~$\U$ in~$\Ima\deltanought\sub C^1$ under~$\bar\delta$. Put more directly,
$$\D := \Ima (\deltanought\circ\gammanought\circ\partial_1) = \{\,\psi_{\partial x}\mid x\in C_1\}$$
(Figure~\ref{cycle-cocyle diagram for sets}).

\begin{figure}[h]
 \hskip30mm
\begin{tikzcd}
\!\!\!\llap{$C_1$}\specrelabove\longrightarrow\pi C_1/\Ker\partialone
\arrow[r, "\bar\partial", "\simeq"'] \arrow[d, "\sigma"', "\simeq", rightarrow]
& \Ima\partialone \rlap{$\ \sub\> C_0 $}
 \arrow[d, "\simeq"', "\tau\ \sub\ \pi\circ\gammanought", rightarrow] 
  \\
\llap{$C^1\> \supseteq\ \Ima\deltanought\> \supe\ \ $}\D
& \U \rlap{$\ \ \sub\  C^0/\Ker\deltanought\ \specrelbelow\longleftarrow\pi\ C^0$}
   \arrow[l, "\simeq", "\bar\delta"']
\end{tikzcd}
\vskip-\medskipamount
 \caption{Cycles, cocycles, and their quotients in hypergraphs}
\label{cycle-cocyle diagram for sets}
\end{figure}%
  \COMMENT{}

\begin{PROP}\label{D}
	The map $\bar\delta\circ\pi\circ\gammanought\circ\bar\partial$ is an isomorphism $C_1/\Ker\partialone\to\D$.%
   \COMMENT{}
\end{PROP}

\proof
  We have to show that $\bar\delta\circ\pi\circ\gammanought\circ\bar\partial$ is injective, i.e., that any $x\in C_1$ which $\deltanought\circ\gammanought\circ\partialone$%
   \COMMENT{}
   maps to~$0\in C^1$ lies in~$\Ker\partialone$ and hence represents zero in~$C_1/\Ker\partialone$. But these are precisely the $x\in C_1$ for which $\psi_{\partial x} = 0\in C^1$, i.e., which are such that $\langle x,y \rangle_\partial = 0$ for all~$y\in C_1$. By Lemma~\ref{lem:innerproduct}, these~$x$ lie in~$\Ker\partialone$.
\endproof

Over the integers, as here, the inclusion $\D\sub\Ima\deltanought$ is usually proper.%
  \COMMENT{}
   In a connected%
   \COMMENT{}
   graph, for example, the 0-cochains~$\varphi_v$ are not only themselves not in~$\gammanought(\Ima\partial_1)$, they do not even represent elements of~$(\pi\circ\gammanought)(\Ima\partial_1)$ in~$C^0/\Ker\delta$.%
  \COMMENT{}
   (They do over the reals.)%
   \COMMENT{}
   Hence the $\bar\deltanought$-images of the classes~$[\varphi_v]$ lie in~$\Ima\deltanought\sm\D$.%
   \COMMENT{}
   On the other hand, we shall see in Lemma~\ref{lem:H1} that $C_1/\Ker\partialone$~is isomorphic not only to~$\D$ (as by Lemma~\ref{D}) but also to the annihilator of~$\Ker\partial_1$ in~$C^1$. This trivially contains $\Ima\delta^0$, and this containment too can be proper. In this case, then, $C_1/\Ker\partialone$ will be isomorphic to two submodules of~$C^1$, one properly contained in~$\Ima\delta^0$ and the other properly containing it.%
   \COMMENT{}
   More on this after Lemma~\ref{lem:H1}.

\medbreak

Let us reflect a little on the difference between the graph and the hypergraph case in this context. In both cases we ended up with an isomorphism~$\sigma$ between $C_1/\Ker\partialone$ and~$\Ima\delta^0$ (or a submodule~$\D$ of~it), and a corresponding isomorphism~$\tau$ between $\Ima\partial_1$ and $C^0/\Ker\deltanought$ (or a submodule~$\U$ of~it).

In the graph case we started out with~$\sigma$, because we had some particularly natural choices for it (one for each spanning tree), and defined $\tau$ simply by completing the diagram, as $\tau = \bar\delta^{-1}\circ\sigma\circ\bar\partial^{-1} \colon\Ima\partial_1\to C^0/\Ker\deltanought$.

In the more general case of hypergraphs there is no natural isomorphism~$\sigma$ between~$C_1/\Ker\partialone$ and $\Ima\delta^0$; indeed, as we shall see in Lemma~\ref{lem:H1} and Example~\ref{MainExample}, there may be none at all.%
   \COMMENT{}
   We therefore started at the other end, with the canonical isomorphism~$\gammanought\colon C_0\to C^0$, and defined $\tau$ explicitly as the restriction of $\pi\circ\gammanought$ to its intended domain of~$\Ima\partialone\sub C_0$. This~$\tau\colon \Ima\partial_1\to C^0/\Ker\deltanought$ is not in general surjective, even for graphs.%
   \COMMENT{}
   So we had to replace $C^0/\Ker\deltanought$ in our diagram with its submodule $\U:= \tau(\Ima\partial_1)$, and correspondingly $\Ima\deltanought\sub C^1$ with its submodule $\D := \bar\deltanought(\U)$. But $\tau$ turned out to be injective. We could therefore define $\sigma$ by completing the diagram, as $\sigma = \bar\delta\circ\tau\circ\bar\partial\colon C_1/\Ker\partialone\to\D\sub\Ima\delta^0$.

Although we have seen that $\U$ and~$\D$ exist as stated, one can still ask for which hypergraphs they have natural interpretations. Conversely, one might try to mimick the situation for graphs and ask for a particularly natural isomorphism~$\sigma$ between $C_1/\Ker\partialone$ and $\Ima\delta^0$ or a natural subspace of it, and then derive $\tau$ from~$\sigma$ rather than the other way round. Or we could seek to find both $\sigma$ and~$\tau$ together:

\begin{PROB}\label{sigmatau}
For which hypergraphs is there a natural pair of isomorphisms $\sigma$ and~$\tau$ that make the diagram of Figure~\ref{cycle-cocyle diagram for graphs} or~\ref{cycle-cocyle diagram for sets} commute?
\end{PROB}

Since graphs are special cases of hypergraphs, Problem~\ref{sigmatau} has a positive answer in some cases. It would be interesting to see to which hypergraphs the particular solution for graphs extends, but also to find new pairs of $\sigma$ and~$\tau$ for other natural hypergraphs.

In Section~\ref{sec:SpTrees} we shall see that, with real coefficients, the standard definition of~$\sigma$ that we know from graphs (see earlier) does extend: one can define `algebraic spanning trees' in hypergraphs that default to spanning trees in graphs and make our earlier approach work in general. So Problem~\ref{sigmatau} has a general positive solution for real coefficients.

\section{Orthogonal decomposition over the reals}\label{sec:OrthogonalReals}

In this section and the next we briefly return to real coefficients. With these everything is much simpler, since our chain groups become vector spaces and we can use linear algebra.%
   \COMMENT{}
   Indeed, all we do in this section can be done in a general linear algebra setting. We shall note this more formally in Section~\ref{sec:LA}, to shed a perhaps unexpected light on orthogonality in real vector spaces as seen through a graph-theoretic lens. For now, however, let us stick to hypergraphs.

Let $\langle\ ,\ \rangle$ denote the standard inner product in~$C_1$ with respect to its basis~$E$,%
   \COMMENT{}
   which maps $e,e'\in E$ to $\langle e,e' \rangle := \delta_{ee'}$ and extends bilinearly to~$C_1\times C_1$. As our isomorphism $\gammaone\colon C_1\to C^1$ was also defined in terms of~$E$ by setting $\gammaone(e)(e') := \delta_{ee'}$, we thus have
 $$\gammaone(x)(y) = \langle x,y \rangle = \langle y,x\rangle = \gammaone(y)(x)$$
 for all $x,y\in C_1$. On~$C^1$ we consider the analogous standard inner product with respect to its basis~$\gamma(E)$.%
   \COMMENT{}
   We can thus speak of orthogonality between subspaces of~$C_1$ or of~$C^1$ now.

In the case of graphs, we have the following well-known relationship between their cycle and bond spaces in $\E = C_1$:
 $$\C = \B^\bot\qquad{\rm and}\qquad \B = \C^\bot.$$%
   \COMMENT{}
 Over real coefficients, as we are considering now, this implies that $\C$~and~$\B$ are (orthogonal) complements in the edge space, i.e.\ that $\E = \C\oplus\B$.%
   \footnote{This need not hold over finite fields.}%
   \COMMENT{}
   This translates into homological language, and thereby to our more general setting, as follows.

The orthogonal complement~$X^\bot$ of a subspace $X\sub C_1$ corresponds via~$\gammaone$ to the annihilator of~$X$ in~$C^1$, the subspace
 $$\gammaone(X^\bot) = \{\,\psi\in C^1\mid \psi(X) = 0\,\}$$
 of~$C^1$.%
   \COMMENT{}
   For $X = \Ker\partial_1$, its annihilator $\gammaone((\Ker\partial_1)^\bot)$ trivially contains~$\Ima\delta^0$. In fact, it is equal to it~-- which is the the essence of the orthogonal decomposition theorem for `cycles and cuts' in an abstract context:

\begin{LEM}\label{LAexactsequences}
   With real coefficients, all hypergraphs satisfy the following:
   \vskip-\medskipamount\vskip0pt
   \begin{enumerate}[\rm (i)]\itemsep=0pt
   \item $\Ker\partial_1$ and $\gammaone^{-1}(\Ima\delta^0)$ are orthogonal complements in~$C_1;$\label{orthocompchains}%
   \COMMENT{}
   \item $\gammaone(\Ker\partial_1)$ and $\Ima\delta^0$ are orthogonal complements in~$C^1$.\label{orthocompcochains}
   \end{enumerate}
\end{LEM}

\begin{figure}[hbt]
\center
\vskip-\medskipamount
\begin{tikzcd}
0
\arrow[r] 
& \Ker\partial_1
\arrow[r, " i"] \arrow[d, "\simeq"]%
   \COMMENT{}
& C_1
\arrow[r, "\partial_1"] \arrow[d, "\gamma_1"', "\simeq"]
& C_0
\arrow[d, "\simeq"', "\gamma_0"]
\\
0 
& (\Ker\partial_1)^*\arrow[l]
& C^1 \arrow[l, "i^*"]
& C^0 \arrow[l, "\delta^0"]
\end{tikzcd}
\vskip-\medskipamount
\caption{Dual exact sequences of vector spaces}%
   \COMMENT{}
\label{LAexact}
\vskip-\medskipamount
\end{figure}

\proof
 The diagram	 in Figure~\ref{LAexact} shows an exact sequence together with its dual, where $(\Ker\partial_1)^* = {\rm Hom}(\Ker\partial_1,\R)$ and $i$ is the inclusion embedding of $\Ker\partial_1$ in~$C_1$. The kernel of its dual map~$i^*$, by definition, contains exactly those $\psi\in C^1$ that send $\Ker\partial_1$ to~0: it is the annihilator%
  \COMMENT{}
  of~$\Ker\partial_1$ in~$C^1$. As duals of exact sequences of vector spaces are exact,%
   \COMMENT{}
   we thus have $\Ima\delta^0 = \Ker i^* = \gammaone((\Ker\partial_1)^\bot)$.%
   \COMMENT{}
   Applying the isomorphism~$\gammaone^{-1}$ now proves~\eqref{orthocompchains}.%
   \COMMENT{}

Statement~\eqref{orthocompcochains} follows from~\eqref{orthocompchains} by applying the isomorphism~$\gammaone\colon C_1\to C^1$, because orthogonality in \eqref{orthocompchains} and~\eqref{orthocompcochains} is computed with respect to the bases $E$ of~$C_1$ and $\gamma(E)$ of~$C^1$.%
   \COMMENT{}
\endproof

For all its brevity, the above proof of Lemma~\ref{LAexactsequences} does not give as much insight as we have it for graphs, where the fundamental cycles and cuts with respect to any fixed spanning tree form explicit bases of $\C$ and~$\B$. These give us an even shorter proof of the orthogonal decomposition lemma. Indeed, as fundamental cycles and cuts are clearly orthogonal to each other, the spaces $\C$ and~$\B$ they generate%
   \COMMENT{}
   are contained in each others orthogonal complement in~$\E$.%
   \COMMENT{}
   But as the cuts are indexed by the tree edges while the cycles are indexed by the tree's chords, the dimensions of $\B$ and~$\C$ add up to that of~$\E$. So this containment cannot be proper: we must have $\C=\B^\bot$ and $\B=\C^\bot$, and in particular $\E=\C\oplus\B$.%
   \COMMENT{}

Prima facie, this works only for graphs (and not over the integers.)%
   \COMMENT{}
   However we can axiomatise fundamental cycles and cuts in a way that enables us to copy the above short proof also for arbitrary hypergraphs~-- indeed for arbitrary finite-dimensional real vector spaces; see Section~\ref{sec:LA}~-- whenever these axioms are met.%
   \COMMENT{}
   Let us do this next.%
   \COMMENT{}

\section{Algebraic spanning trees of hypergraphs}\label{sec:SpTrees}

Let us return to our general hypergraph setup, with real coefficients. To help our intuition along, let us write $\C:=\Ker\partial_1$ and $\B:=\gammaone^{-1}(\Ima\delta^0)$, as well as~$\E:= C_1$. Let us call a subset~$T$ of~$E$ an (algebraic) {\em spanning tree\/} of our hypergraph,
 with {\em chords\/} $e\in E\sm T$, if it satisfies the following two conditions:%
  \COMMENT{}
  \begin{itemize}
  \item $\B$ has a basis $(\,x_t\mid t\in T\,)$ such that $\langle x_t, t'\rangle = \delta_{tt'}$%
   \COMMENT{}
  for all $t'\in T$;%
   \COMMENT{}
\item $\C$ has a basis $(\,x_e\mid e\in E\sm T\,)$ such that $\langle x_e, e'\rangle = \delta_{ee'}$ for all $e'\in E\sm T$.%
   \COMMENT{}
  \end{itemize}
Let us call the~$x_t$ with $t\in T$ the {\em fundamental cuts\/}, and the $x_e$ with $e\in E\sm T$ the {\em fundamental cycles\/}, of our hypergraph with respect to~$T$.%
   \COMMENT{}
   We shall refer to the above two statements as the {\em spanning tree axioms\/}.%
   \COMMENT{}
   The spanning tree is said to be {\em over\/} the coefficients used in the homology on which $\B$ and~$\C$ are based.

Note that still $\C\sub\B^\bot$ and $\B\sub\C^\bot$, because the annihilator of~$\Ker\partial_1$ contains~$\Ima\delta^0$. Hence if $\B$ and~$\C$ have bases indexed by complementary subsets of~$E$, as demanded in the spanning tree axioms, those two inclusions must hold with equality,%
   \COMMENT{}
   and $(\,x_r\mid r\in E\,)$ will be a basis of $\E=\C\oplus\B$.%
  \COMMENT{}

Fundamental cycles and cuts in graphs are clearly an example of this. In fact, when our hypergraph is a (connected) graph, they are the only example: given bases $(\,x_t\mid t\in T\,)$ of~$\B$ and $(\,x_e\mid e\in E\sm T\,)$ of~$\C$ satisfying our two conditions, where $T$ is only assumed to be a subset of~$E$, one can show%
   \COMMENT{}
   that $T$ is the edge set of a spanning tree whose fundamental cuts and cycles are these $x_t$ and~$x_e$.

\begin{THM}\label{spanningtrees}
Every hypergraph has an algebraic spanning tree over the reals.
\end{THM}%
   \COMMENT{}

\proof
Let $T\sub E$ be any maximal subset of~$E$ such that $\partialone T = (\,\partialone t\mid t\in T\,)$%
   \COMMENT{}
   is linearly independent in~$C_0$. Then for every $e\in E\sm T$ there is a unique%
   \COMMENT{}
   $x'_e\in C_1$ in the span of~$T$ such that $\partialone e = \partialone x'_e$;%
   \COMMENT{}
   let
 $$x_e := e - x'_e.$$
 Then $\partialone x_e = 0$, so $x_e\in\C = \Ker\partial_1$.

The family $(\,x_e\mid e\in E\sm T\,)$ is linearly independent.%
   \COMMENT{}
   Indeed, as $x_e = e - x'_e$ with $x'_e$ spanned by~$T$, the only element of $E\sm T$ in the representation of~$x_e$ over~$E$ whose coefficient is non-zero is~$e$. We can therefore have $\sum_{e\in E\sm T} \lambda_e x_e = 0$ only with $\lambda_e = 0$ for all~$e\in E\sm T$.%
   \COMMENT{}

The family $(\,x_e\mid e\in E\sm T\,)$ also generates~$\C$. Indeed, as it is linearly independent we just have to show that $\dim\C = \dim\Ker\partial_1 = |E\sm T|$. As $\dim\E = |E|$ and $T\sub E$, this amounts to showing that $\dim\Ima\partial_1 = |T|$.%
   \COMMENT{}
   Since $E$ generates~$\E$, clearly $\partial E$ generates~$\Ima\partial_1$. By the maximality of~$T$ this implies that~$\partial T$, too, generates~$\Ima\partial_1$.%
   \COMMENT{}
   As $\partial T$ is also linearly independent (by definition of~$T$) it is therefore a basis of~$\Ima\partial_1$, so $\dim\Ima\partial_1 = |T|$ as desired.

To complete our proof of the second spanning tree axiom it remains to show that $\langle x_e,e'\rangle = \delta_{ee'}$ for all $e,e'\in E\sm T$. By definition of~$x_e$ we have $x_e = e + \sum_{t\in T}\lambda_t t$ for some suitable coefficients~$\lambda_t$. Then
 $$\textstyle \langle x_e, e'\rangle = \langle e,e'\rangle + \sum_{t\in T} \lambda_t \langle t,e'\rangle = 
  \delta_{ee'} + \sum_{t\in T} \lambda_t \delta_{te'} = \delta_{ee'}\,,$$
 since $\langle e,e'\rangle = \delta_{ee'}$ and $\langle t,e'\rangle = \delta_{te'} = 0$ as $e'\notin T\owns t$. This completes our proof of the second spanning tree axiom for $(\,x_e\mid e\in E\sm T\,)$.

Let us now prove the first spanning tree axiom for the family $(\,x_t\mid t\in T\,)$ whose~$x_t$ are given via their duals~$\gamma(x_t)\in C^1$ as follows:%
   \COMMENT{}
   $$\gamma(x_t)\>\colon  \begin{cases}\ t'\mapsto\delta_{tt'} & \text{for all}\ t'\in T\\
   \ e\mapsto -\gamma(x_e)(t) & \text{for all}\ e\in E\sm T.
   \end{cases}$$
 Note that $\langle x_t,t'\rangle = \delta_{tt'}$ for all $t,t'\in T$ by definition,%
   \COMMENT{}
   so all we need to show is that $(\,x_t\mid t\in T\,)$ is a basis of~$\B$. The bulk of this is to show that the~$x_t$ lie in~$\B = \gammaone^{-1}(\Ima\delta^0)$, so let us do this first.

As $\Ima\delta^0$ equals,%
   \COMMENT{}
   by Lemma~\ref{LAexactsequences}, the annihilator of~$\C$, it suffices to show that every $\gammaone(x_t)$ lies in that annihilator, i.e., sends~$\C$ to zero. By the first part of our proof it suffices to show this for our basis $(\,x_e\mid e\in E\sm T\,)$ of~$\C$. We thus have to show that $\langle x_t,x_e\rangle = 0$ for all $t\in T$ and~$e\in E\sm T$.%
   \COMMENT{}
   As earlier, we have $x_e = e + \sum_{t'\in T}\lambda_{t'} t'$ for suitable coefficients~$\lambda_t'$.%
   \COMMENT{}
Then
 	\begin{eqnarray*}
	\langle x_t,x_e\rangle &=& \langle x_t, e\rangle + \textstyle\sum_{t'\in T} \lambda_{t'} \langle x_t, t'\rangle \\
 & =& \gamma(x_t)(e) + \lambda_t  \\
 & =& -\gamma(x_e)(t) + \lambda_t  \\
 & =& -\big(\gamma(e)(t) + \textstyle\sum_{t'\in T} \lambda_{t'} \gamma(t')(t) \big) + \lambda_t\\
 & =& 0%
   \COMMENT{}
 \end{eqnarray*}\vskip-.5\baselineskip\noindent
 as desired.

The proof that $(\,x_t\mid t\in T\,)$ is a basis of~$\B$ is similar to our earlier proof for~$\C$. For linear independence we once more use the fact that the only $t'\in T$ whose coefficient in the representation of~$x_t$ over~$E$ is non-zero is $t$ itself, this time by the explicit definition of~$\gammaone(x_t)$.%
   \COMMENT{}
   As
 $$\dim\B = \dim\E-\dim\C = |E| - |E\sm T| = |T|$$
 by Lemma~\ref{LAexactsequences} and the first half of our proof, this shows that $(\,x_t\mid t\in T\,)$ is in fact a basis of~$\B$.
 \endproof

Our proof of Theorem~\ref{spanningtrees}%
   \COMMENT{}
   gives a positive answer to Problem~\ref{sigmatau} for real coefficients. In fact,%
   \COMMENT{}
   it provides a natural isomorphism $\sigma\colon \E/\C\to \gammaone(\B)$, which makes our earlier diagram of Figure~\ref{cycle-cocyle diagram for graphs} commute when $\tau\colon\Ima\partial_1\to C^0/\Ker\delta$ is taken to be the induced isomorphism $\bar\delta^{-1}\circ\sigma\circ\bar\partial^{-1}$:%
   \COMMENT{}

\begin{COR}\label{SpTreeCor}
Every hypergraph has an algebraic spanning tree%
   \COMMENT{}
   over the \hbox{reals} such that mapping the classes $t+\C$ to its fundamental cuts~$x_t$ defines an isomorphism $\E/\C\to\B$.
\end{COR}

\begin{proof}
Let $T$ be the spanning tree defined in the proof of Theorem~\ref{spanningtrees}, i.e.,\ such that $\partial T$ is a basis of~$\Ima\partial_1$. Then the classes $[t]$ with $t\in T$ form a basis of~$\E/\C$.%
   \COMMENT{}
   By the first spanning tree axiom, mapping $[t]$ to~$x_t$ defines a bijection between bases of $\E/\C$ and of~$\B$.%
   \COMMENT{}
\end{proof}

Note that the mere fact that $\E/\C$ and~$\B$ are isomorphic already follows from the standard orthogonal decomposition theorem, Lemma~\ref{LAexactsequences}. What Theorem~\ref{spanningtrees} adds is a particularly natural isomorphism witnessing this, as we know it from graphs.

Our proof of Theorem~\ref{spanningtrees} was based on explicit definitions of a spanning tree, its fundamental cycles, and its fundamental cuts. Using matroid representation theory one can show that this was not just one choice amongst others, but that all these have to arise in this way: that $\partial T$ must be a basis of~$\Ima\partial_1$, and that the $x_e$ and~$x_t$ are then exactly as we defined them (relative to~$T$)~\cite{BowlerPC20, OxleyBook}.%
   \COMMENT{}
   In particular, Corollary~\ref{SpTreeCor} holds for every algebraic spanning tree over the reals, not just for a particular one.

\section{Spanning trees of vector spaces}\label{sec:LA}

The purpose of this section is to record briefly what we proved in Theorem~\ref{spanningtrees} without saying so: that orthogonal decomposition in arbitrary finite-dimensional real vectors spaces can be described in graph-theoretical terms, with reference to the `fundamental cycles and cuts' of a `spanning tree'.%
   \COMMENT{}

Instead of our boundary homomorphism $\partialone\colon C_1\to C_0$ let us consider any homomorphism $f\colon V\!\to W$ between finite-dimensional real vector spaces, replacing the coboundary homomorphism~$\delta^0$ with its dual $f^*\colon W^*\to V^*$ accordingly.

Let $S$ be any basis of~$V\!$. We then have the canonical isomorphism $V\!\to V^*$ which sends each $s\in S$ to the unique linear form in~$V^*$ that maps all the $s'\in S$ to~$\delta_{ss'}$; this takes the place of our earlier~$\gamma_1$. And when we speak of orthogonality in~$V\!$, the inner product referred to is that for which $\langle s,s'\rangle = \delta_{ss'}$ for all~$s,s'\in S$: the canonical inner product of the coefficient tuples over~$S$.%
   \COMMENT{}

There is one more simplification. Instead of considering arbitrary homomorphisms $f\colon V\!\to W$ as above, there is no loss of generality in considering the canonical homomorphism $v\mapsto [v]$ from $V\!$ to~$V/\Ker f$. The simplification here is not so much that we replace $\Ima f\sub W$ with its isomorphic copy~$V/\Ker f$, but that the subspaces $\C = \Ker f$ and $\B = (\Ima f^*)^*$%
   \COMMENT{} 
   of~$V\!$ depend only on the kernel of~$f$ rather than on $f$ itself.%
   \footnote{Since every homomorphism $\Ima f\to\R$ extends to a homomorphism $W\!\to\R$, the subspace $\Ima f^*$ of~$V^*$ remains the same if we replace $W$ with its subspace~$\Ima f$.}%
   \COMMENT{}
   Since the spanning trees we defined in Section~\ref{sec:SpTrees} depend only on the choice of our basis~$S$ of~$V\!$ and on its subspaces $\C$ and~$\B$,%
   \COMMENT{}
   this means that in fact they only depend on $S$ and the subspace $U=\Ker f$ of~$V\!$.

For this reason let us say that, given~$V\!$, a~basis $S$ of~$V\!$, and a subspace $U\sub V\!$, a set $T\sub S$ is a {\em spanning tree\/} of~$V\!$ {\em with respect to}~$U$ (and~$S$) if it satisfies our two spanning tree axioms:

\goodbreak

  \begin{itemize}
  \item $U^\bot$%
     \COMMENT{}   
   has a basis $(\,x_t\mid t\in T\,)$ such that $\langle x_t, t'\rangle = \delta_{tt'}$ for all $t'\in T$;
\item $U$ has a basis $(\,x_s\mid s\in S\sm T\,)$ such that $\langle x_s, s'\rangle = \delta_{ss'}$ for all $s'\in S\sm T$.
  \end{itemize}

  \noindent
   The~$x_t$ are the {\em fundamental cuts\/} of~$T$, the~$x_s$ its {\em fundamental cycles\/}.

\begin{THM}
Let $V\!$ be a finite-dimensional real vector space with basis~$S$, and $U\sub V$ any subspace. Then $S$ contains a spanning tree of~$V\!$ with respect to~$U\!$.
Its fundamental cycles form a basis of~$U\!$, its fundamental cuts a basis of~$U^\bot\!$.\looseness=-1
\end{THM}

\proof
   Redo the proof of Theorem~\ref{spanningtrees} with the canonical projection ${\pi\colon V\!\to V/U}$ replacing $\partialone\colon C_1\to C_0$,%
   \COMMENT{}
   and its dual $\pi^*$ replacing~$\delta^0$. Then $\C = \Ker\pi = U$, and $\B = (\Ker\pi)^\bot = U^\bot$ by well-known linear algebra%
     \COMMENT{}   
   or Lemma~\ref{LAexactsequences}\,\eqref{orthocompchains}. The spanning tree found in that proof, therefore, is as required here.
\endproof

\section{Algebraic spanning trees over the integers}\label{sec:SpanningTreesIntegers}

Let us briefly address what happens in the context of Section~\ref{sec:SpTrees} when we take integer coefficients.%
   \COMMENT{}
   Let us write $\C_\Z := \Ker\partial_1$ and $\B_\Z := \gammaone^{-1}(\Ima\delta^0)$ for the cycles and%
   \COMMENT{}
   coboundaries when our chain groups are taken over the integers, and $\C_\R$ and~$\B_\R$ when they are taken over the reals. Note that $\partialone$ in the former case is the restriction to chains with inter coefficients of the $\partialone$ for real coefficients, since both expand linearly from the map~$\partialone\colon E\to \{-1,0,1\}$. In particular, $\C_\Z\sub\C_\R$. Similarly, we have $\B_\Z\sub\B_\R$,%
   \COMMENT{}
   since every 0-cochain over the integers extends to one over the reals.%
  \COMMENT{}

Staying entirely within the integer coefficient scenario, we can still define spanning trees as we did earlier;%
   \COMMENT{}
  we would then require $(\,x_t\mid t\in T\,)$ to be a $\Z$-module basis of~$\B_\Z$, and $(\,x_e\mid e\in E\sm T\,)$ to be one of~$\C_\Z$.

But we can also combine the two scenarios, as follows. Call a spanning tree~$T$ over the reals {\em integral\/} if the corresponding%
   \COMMENT{}
   bases $(\,x_t\mid t\in T\,)$ of~$\B_\R$ and $(\,x_e\mid e\in E\sm T\,)$ of~$\C_\R$ satisfy $x_t\in\B_\Z$ for all $t\in T$, and $x_e\in\C_\Z$ for all $e\in E\sm T$. In particular, the $x_e$ and the~$x_t$ then have integer coefficients in their representations over~$E$. In the case of the~$x_e$ this is equivalent to $x_e\in\C_\Z$, but in the case of the~$x_t$ we are asking a little more: we also require that $\gamma(x_t)$ is an image under~$\delta^0$ of a 0-cochain with integer coefficients.%
   \COMMENT{}

Integral spanning trees may not be of immediate interest in their own right, but they are a useful technical device. Via Lemma~\ref{sp.tree bases} below they can help us prove that a spanning tree over the integers exists. Conversely, a spanning tree over the integers may be useful in applications only once we know it is also an integral spanning tree over the reals;%
   \COMMENT{}
   see Proposition~\ref{1.9.5} for an example.%
   \COMMENT{}

\begin{LEM}\label{sp.tree bases}
The spanning trees of a hypergraph over the integers are precisely its integral spanning trees over the reals.
\end{LEM}%
   \COMMENT{}

\proof
   For the backward implication, let $T\sub E$ be an integral spanning tree over the reals. Its associated bases $(\,x_t\mid t\in T\,)$ of~$\B_\R$ and $(\,x_e\mid e\in E\sm T\,)$ of~$\C_\R$ satisfy the spanning tree axioms, which are otherwise%
   \COMMENT{}
   independent of which coefficients are considered.%
   \COMMENT{}
   By definition of an integral spanning tree, the $x_t$ lie in~$\B_\Z$ and the~$x_e$ lie in~$\C_\Z$. They are clearly linearly independent also over~$\Z$. It remains to show that they generate $\B_\Z$ and~$\C_\Z$, respectively, with integer coefficients.

Any $x\in\C_\Z\sub\C_\R$ has two expansions: as $x = \sum_{e\in E\sm T} \lambda_e x_e$ with real coefficients~$\lambda_e$, and as $x = \sum_{e\in E} n_e e$ with integer coefficients~$n_e$.%
   \COMMENT{}
   This means that, in fact, $\lambda_e = n_e$ for all $e\in E\sm T$: since $\langle x_{e'}, e\rangle = \delta_{ee'}$ by the second spanning tree axiom, the term $\lambda_{e'} x_{e'}$ for $e'\ne e$ has coefficient~0 at~$e$ in its own expansion over~$E$, so the coefficient of $\lambda_e x_e$ over~$E$, which is~$\lambda_e$,%
   \COMMENT{}
   must equal the coefficient of~$x$ at~$e$, which is~$n_e$.%
   \COMMENT{}
   Thus, $x$~has an expansion over $(\,x_e\mid e\in E\sm T\,)$ with integer coefficients, as desired.

It remains to show that every $x\in\B_\Z$ has an expansion with integer coefficients over $(\,x_t\mid t\in T\,)$. As before, $x\in\B_\Z\sub\B_\R$ has two expansions: as $x = \sum_{t\in T}\mu_t x_t$ with real coefficients~$\mu_t$, and as $x = \sum_{e\in E} m_e e$ with integer coefficients~$m_e$. As $\langle x_{t'},t\rangle = \delta_{tt'}$ by the first spanning tree axiom, we see that in fact $\mu_t = m_t$ for all $t\in T$ (as earlier for~$\C_\Z$). Thus, $x$~has an expansion over $(\,x_t\mid t\in T\,)$ with integer coefficients, as desired.

For the forward implication let $T\sub E$ be a spanning tree over the integers with associated module bases $(\,x_t\mid t\in T\,)$ of~$\B_\Z$ and $(\,x_e\mid e\in E\sm T\,)$ of~$\C_\Z$. We shall prove that these subsets of $\B_\Z\sub\B_\R$ and $\C_\Z\sub\C_\R$ generate $\B_\R$ and~$\C_\R$ with real coefficients. As their cardinalities add up to the dimension~$|E|$ of the real vector space~$C_1$, which $\B_\R$ and $\C_\R$ generate by Lemma~\ref{LAexactsequences}, this will show that they are linearly independent too, making them bases of $\B_\R$ and~$\C_\R$. Since they satisfy the spanning tree axioms over~$\R$ as they do over~$\Z$, this will establish $T$ as an integral spanning tree over the reals.

Let us start by showing that, with real coefficients, $(\,x_t\mid t\in T\,)$ generates~$\B_\R$.\penalty-200\ Any $x\in\B_\R$ has%
   \COMMENT{}
   the form $x = (\gammaone^{-1}\circ\deltanought\circ\gammanought)(y)$ for some $y\in C_0$ with real coefficients, where the two isomorphisms~$\gamma$%
   \COMMENT{}
   and the coboundary homomorphism~$\deltanought$ are those of the homology of our hypergraph over the reals. Since $V\!$ is a basis of~$C_0$, it suffices to show that its images $x_v := (\gammaone^{-1}\circ\deltanought\circ\gammanought)(v)$ in~$\B_\R$ under this same map, for all~$v\in V\!$, are generated by the~$x_t$ with real coefficients.

For each $v\in V\!$, however, its image $z_v := (\gammaone^{-1}\circ\deltanought\circ\gammanought)(v)$ in~$\B_\Z$ taken in the homology over~$\Z$ is just our earlier~$x_v$, since applying the boundary homomorphism $\partialone$ to the elements of~$E$ yields the same 0-chain regardless of which coefficients we are considering in our homology.%
   \COMMENT{}

Now as $x_v$ equals $z_v\in\B_\Z$, it is generated (even with integer coefficients) by the basis $(\,x_t\mid t\in T\,)$ of~$\B_\Z$, as was our aim to show.

It remains to show that $(\,x_e\mid e\in E\sm T\,)$ generates~$\C_\R$ with real coefficients. Let $x\in\C_\R$ be given. If all its coefficients are rational, we can find $n\in\Z$ such that $nx\in\C_\Z$, which thus has an expansion over the~$x_e$. Dividing the (integer) coefficients of this expansion by~$n$ yields the desired expansion of~$x$ over the~$x_e$ with real (indeed, rational) coefficients.

For the general case of real coefficients let us think of~$C_1$%
   \COMMENT{}
   as the real vector space~$\R^{|E|}$, and of $\C_\R = \Ker\partialone$ as its subspace~$U$. The real span of the~$x_e$%
   \COMMENT{}
   is a subspace of~$\R^{|E|}$, and hence closed. It is also a subset of~$U$ which, as we have seen, contains its rational points. As these are dense in~$U$,%
   \COMMENT{}
    it equals~$U$.
   \endproof

The hypergraph in Example~\ref{MainExample} will have no spanning tree over the integers.

\begin{PROB}\label{prob:sp.trees}
Which hypergraphs have a spanning tree over the integers?
\end{PROB}

\section{Orthogonal decomposition over the integers}\label{sec:OrthogonalIntegers}

Let us return to the topic of orthogonal decomposition, which motivated us to axiomatise spanning trees for hypergraphs over the reals.
With integer coefficients we no longer have orthogonal decomposition as in Lemma~\ref{LAexactsequences}, even for graphs: although $\C:=\Ker\partial_1$ and $\B:=\gammaone^{-1}(\Ima\delta^0)$ still intersect only in zero (unlike with $\mathbb F_2$ coefficients, say), they need not span~$\E$ between them. To see this, consider a triangle over~$\Z$,%
   \COMMENT{}
   or just two vertices joined by two parallel edges. We shall discuss this example in more detail after the proof of Theorem~\ref{Thm11}.

We do still have $\C = \B^\bot$; see Proposition~\ref{1.9.5} below. The other possible translation of Lemma~\ref{LAexactsequences}, that $\B = \C^\bot$, is trickier. Of course, we still have $\B\sub\C^\bot$ by definition of~$\B$. But this inclusion can be strict, as we shall see in Lemma~\ref{lem:H1} and Example~\ref{MainExample}.%
   \COMMENT{}

So the general problem of when $\B = \C^\bot$ remains open:

\begin{PROB}\label{BisCbot}
Which hypergraphs satisfy $\B = \C^\bot$ over the integers?%
   \COMMENT{}
\end{PROB}

Hypergraphs with spanning trees over the integers are among these:

\begin{PROP}\label{1.9.5}
\begin{enumerate}[\rm(i)]\itemsep=0pt
   \item All hypergraphs satisfy~$\C=\B^\bot$ over the integers.\label{195C}
   \item Hypergraphs with a spanning tree over the integers satisfy $\B=\C^\bot\!$ over the integers.\label{195B}%
   \COMMENT{}
\end{enumerate}	
\end{PROP}

\proof
As $\psi(x)=0$ for all $\psi\in\Ima\delta^0$ and $x\in\Ker\partial_1$, we have $\C\sub\B^\bot$ and ${\B\sub\C^\bot}$ in both cases, \eqref{195C} and~\eqref{195B}.%
   \COMMENT{}

\eqref{195C} For a proof of~$\B^\bot\sub\C$ we have to show that any $x\in C_1$ which every $\delta(\varphi)$ with $\varphi\in C^0$ sends to $0\in\Z$ lies in~$\Ker\partial$. But this follows from Lemma~\ref{lem:innerproduct}: for any $x\in C_1\sm\Ker\partial$, choosing  $\varphi := \gammanought(\partial x)$ yields $\delta(\varphi)(x) = \varphi(\partial x) = \langle x,x\rangle_\partial\ne 0$. 

  \eqref{195B} By Lemma~\ref{sp.tree bases}, our hypergraph also has an integral spanning tree~$T$ over the reals; let $(\,x_t\mid t\in T\,)$ and $(\,x_e\mid e\in E\sm T\,)$ be its associated bases of~$\B_\R$ and~$\C_\R$. As all the~$x_t$ lie in~$\B_\Z = \B$, by the definition of integral spanning trees, it suffices for a proof of~$\C^\bot\sub\B$ to show that every $x\in\C^\bot$ has an expansion over $(\,x_t\mid t\in T\,)$ with integer coefficients.

Since all the~$x_e$ lie in~$\C_\Z = \C$,%
   \COMMENT{}
   we have $\langle x,x_e\rangle = 0$ for all $e\in E\sm T$. But the~$x_e$ generate~$\C_\R$ over the reals, so this implies that $x\in\C_\R^\bot$ too. By Lemma~\ref{LAexactsequences} we have $\C_\R^\bot = \B_\R$, so~$x$ has an expansion $x = \sum_{t\in T}\mu_t x_t$ with real coefficients~$\mu_t$. By definition,%
   \COMMENT{}
   $x$~also has an expansion $x = \sum_{e\in E} m_e e$ with integer coefficients~$m_e$. As $\langle x_{t'},t\rangle = \delta_{tt'}$ by the first spanning tree axiom, we see that in fact $\mu_t = m_t$ for all $t\in T$ (as in the proof of Lemma~\ref{sp.tree bases}). Thus, $x$~has an expansion over $(\,x_t\mid t\in T\,)$ with integer coefficients, as desired.
   \endproof

Theorem~\ref{Thm5} offers an algebraic characterisation of the hypergraphs that satisfy $\B=\C^\bot$ over the integers. However it would be interesting to have a combinatorial characterisation too. 

\section{Homology groups}\label{sec:H1}%

Let us now take a closer look at the homology groups of hypergraphs. For graphs, as for hypergraphs, the first homology group~$H_1$ is simply~$\Ker\partial_1$, and the first cohomology group~$H^1$ is~$C^1/\Ima\deltanought$.

In the case of (connected) graphs there are natural isomorphisms%
   \COMMENT{}
   between these, as follows. Fix any spanning tree~$T$. For any (oriented) chord $e=uv$ of~$T$ write $x_e := \sum_{e'\in E}\lambda_{e'} e'$ with $\lambda_{e'} = 1$ for $e'=e$ and all oriented edges~$e'$ on the path in~$T$ from $v$ to~$u$,%
   \footnote{We assume here that these edges~$e'$ are oriented in the direction of this path, from $v$ towards~$u$~-- which, of course, need not be the case. Formally we put $\lambda_{e'} = 1$ if this orientation of~$e'$ happens to be its default orientation, and $\lambda_{e'} = -1$ otherwise.}
   and $\lambda_{e'} = 0$ for all other edges~$e'$. (Thus, $x_e$ sends a flow of~1 round the fundamental cycle of~$e$ and is zero elsewhere.) Then $x_e\in\Ker\partialone$.
   As earlier, write $\psi_e$ for the 1-cochain that sends $e$ to~1 and all other edges to~0.\penalty-200\ These $\psi_e$ represent non-zero classes in~$C^1/\Ima\deltanought$, since any $\psi\in\Ima\deltanought$ that assigns zero to the edges of~$T$ (as the~$\psi_e$ do) has the form $\psi=\varphi\circ\partial$ with $\varphi$ constant on~$V\!$, so it also sends the chords of~$T$ to zero.%
   \COMMENT{}
   In fact, the $x_e$ freely generate~$\Ker\partialone$%
   \COMMENT{}
   and the $[\psi_e]$ freely generate~$C^1/\Ima\deltanought$%
   \COMMENT{}
   as abelian groups or $\Z$-modules. So $x_e\mapsto [\psi_e]$ defines a natural isomorphism $H_1\to H^1$ for graphs.%
   \COMMENT{}
   All this is well known~\cite{Hatcher}.

As we shall see in a moment,%
   \COMMENT{}
   the above construction of an isomorphism ${H_1\to H^1}$ for graphs carries over to hypergraphs with real coefficients, with spanning trees as provided by Theorem~\ref{spanningtrees}.%
   \COMMENT{}
   With integer coefficients things are more complicated, but our results in this section will establish the following:

\begin{THM}
Hypergraphs that have an algebraic spanning tree over the integers satisfy $H_1\simeq H^1\!$ over the integers.	
\end{THM}

\proof
   By Proposition~\ref{1.9.5}\,\eqref{195B}, hypergraphs with an algebraic spanning tree over the integers satisfy $\B=\C^\bot$ over the integers. By Theorem~\ref{Thm5} below, such hypergraphs also satisfy $H_1\simeq H^1$.
   \endproof

   The hypergraph in Example~\ref{MainExample} below is such that $H_1\not\simeq H^1$, but it has no spanning tree over the integers. 

\medbreak

Our main aim in this section is to prove that $H_1\simeq H^1$ if $\partial_1$ satisfies a certain condition. If it does, we shall in fact prove that $H_1$ is {\em canonically\/} isomorphic to~$H^1$ in the sense that, in the terminology of Figure~\ref{LAexact},
 $$H_1 = \Ker\partial_1\simeq (\Ker\partial_1)^* = \Ima i^*\simeq C^1/\Ker i^* = C^1/\Ima\delta^0 = H^1$$
 with the isomorphism $\pi\colon [\psi]\mapsto i^*(\psi) = \psi\restricts\Ker\partial_1$ from $C^1/\Ker i^*$ to~$(\Ker\partial_1)^*$%
   \COMMENT{}
   and any isomorphism $\Ker\partial_1\to (\Ker\partial_1)^*$.%
   \COMMENT{}
   With real coefficients, the line displayed above amounts to a proof of $H_1\simeq H^1$. But with integer coefficients it does not, since $\Ima\delta^0$ can be contained properly in the annihilator~$\Ker i^*$ of~$\Ker\partial_1$ in~$C^1$. However we shall find a condition which characterises the hypergraphs for which these two submodules of~$C^1$ coincide, and then prove that $H_1\simeq H^1$ canonically whenever this condition is met. In order to do this, we need some preparation.

Recall that a submodule $M'$ of a module~$M$ is a {\em direct summand\/} of~$M$ if $M = M'\oplus M''$ for some other module $M''\sub M$, which means that every $m\in M$ has a unique representation as $m = m'+m''$ with $m'\in M'$ and $m''\in M''$. 

\begin{THM}\label{Thm5}
	The following statements are equivalent for hypergraphs $(V,E)$ with integer coefficients:%
   \COMMENT{}
   \begin{enumerate}[\rm (i)]\itemsep=0pt
      \item 	The homology groups $H_1\!$ and~$H^1\!$ of $(V,E)$ are canonically isomorphic.\label{H1=H1}
      \item The annihilator of~$\Ker\partial_1$ in~$C^1\,$equals~$\Ima\delta^0$.\label{A=Im}
      \item $\,\B=\C^\bot$ over the integers (in the notation of Section~\ref{sec:OrthogonalIntegers}).\label{B=Cbot}
      \item The image of~$\partial_1$ is a direct summand of~$C_0$.\label{Imd}\footnote{Note that while $\Ima\partial_1$ is clearly a direct summand of~$C_0$ if it is the span of a subset of~$V\!$, this is by no means necessary. It will be the span of a subset of some basis of~$C_0$,%
   \COMMENT{}
   but this basis need not be~$V\!$.%
   \COMMENT{}}
      \item The map $[\psi]\mapsto \psi\restricts\Ker\partial_1$ is an isomorphism $H^1 \to {\rm Hom}(H_1,\Z)$.\label{HomH1}
   \end{enumerate}
\end{THM}

Theorem~\ref{Thm11}.%
   \COMMENT{}
   will show that graphs satisfy the assertions of Theorem~\ref{Thm5}. Let us call hypergraphs that satisfy these assertions {\it algebraically graph-like\/}.
   Example~\ref{MainExample} shows a simple hypergraph that is not algebraically graph-like.

\medbreak

For our proof of Theorem~\ref{Thm5} we need a few lemmas. The first is the elem\-entary divisor theorem for submodules of free modules%
  \COMMENT{}
   over principal ideal domains~\cite{BoschLA}. For $\Z$-modules it says the following:

\begin{LEM}\label{ETS}
	Let $M'$ be a submodule of a free $\Z$-module~$M$ of finite rank. Then $M$ has a basis $(x_1,\dots,x_m)$ such that $M'$ has a basis $(n_1 x_1, \dots, n_k x_k)$ with $k\le m$ and $n_i | n_{i+1}$ for all $i=1,\dots,k-1$.
\end{LEM}

\noindent
   The integers $n_1,\dots,n_k$ are the {\em elementary divisors\/} of~$M'$ in~$M$. Up to their sign%
   \COMMENT{}
   they are uniquely determined by~$M'$ and~$M$; in particular, they do not depend on the choice of $(x_1,\dots,x_m)$.

\medbreak

One of the many consequences of Lemma~\ref{ETS} is that submodules of free $\Z$-modules of finite rank are also free. Here are some more consequences that we shall use repeatedly:

\begin{LEM}\label{Lalpha}
	The following statements are equivalent for submodules~$M'$ of a free $\Z$-module~$M\!$ of finite rank:
   \begin{enumerate}[{\rm (i)}]\itemsep=0pt\vskip-\smallskipamount\vskip0pt
   \item all the elementary divisors of~$M'$ are~$\pm 1$;
\item $M'$ is a direct summand of~$M$;
\item every homomorphism $\varphi'\colon M'\to\Z$ extends to a homomorphism ${\varphi\colon M\to\Z}$.
   \end{enumerate}
\end{LEM}

\begin{proof}
 Consider bases of $M$ and~$M'$ as in Lemma~\ref{ETS}.

	(i)$\to$(ii) If (i) holds then $(n_1 x_1, \dots, n_k x_k,\, x_{k+1},\dots, x_m)$ too is a basis of~$M$, which proves~(ii).%
   \COMMENT{}

(ii)$\to$(iii) is clear.%
   \COMMENT{}

(iii)$\to$(i) Suppose (i) fails; then the `largest' elementary divisor $n_k$ of~$M'$ is not a unit: $n_k\ne\pm1$. Define a homomorphism $\varphi'\colon M'\to\Z$ by mapping its base elements $n_k x_k$ to~1 and $n_i x_i$ to~0 for all $i < k$. Now consider any homomorphism $\varphi\colon M\to\Z$; we show that $\varphi\restricts M'\ne\varphi'$. Indeed if $\varphi(x_k) = 0$ then $\varphi(n_k x_k) = 0\ne 1 = \varphi'(n_k x_k)$. But if $\varphi(x_k)\in\Z\sm\{0\}$ then, as $n_k\ne\pm 1$, we have $\varphi(n_k x_k) = n_k\varphi(x_k)\ne 1 = \varphi'(n_k x_k)$.
\end{proof}

\begin{LEM}\label{Lemma7}
	The kernel of a homomorphism $f\colon M\to N$ between two free $\Z$-modules of finite rank is always a direct summand of~$M$.
\end{LEM}

\proof
   Consider bases of $M$ and~$M':=\Ker f$ as in Lemma~\ref{ETS}. If $M'$ is not a direct summand of~$M$ then, by Lemma~\ref{Lalpha}, its `largest' elementary divisor is not a unit: $n_k\ne\pm1$. Since $(n_1 x_1,\dots,n_k x_k)$ is a basis of~$M'$, and $n_k x_k\in M'$ has a unique representation over this basis, this means that $x_k\notin M'$.%
   \COMMENT{}
  So $f(n_k x_k) = 0$ since $n_k x_k\in\Ker f$, while $f(x_k)\ne 0$ since $x_k\notin\Ker f$.

\goodbreak

Putting the two together we have $n_k f(x_k) = f(n_k x_k) = 0$ but $f(x_k)\ne 0$, which means that $f(x_k)$ has torsion in~$N$ as $n_k\ne 0$. This contradicts our assumption that $N$ is free.%
   \COMMENT{}
   \endproof

Note that Lemma~\ref{Lemma7} has no analogue for the image (rather than the kernel) of a homomorphism $f\colon M\to N$ as above, which need not be a direct summand of~$N$. For example, consider $M = N = \Z$ with $f\colon a\mapsto 3a$.%
   \COMMENT{}
   This little fact lies at the heart of all the differences between integer and real coefficients in our context. More about this after our proof of Theorem~\ref{Thm5}.

\begin{LEM}\label{lem:H1}   Every hypergraph~$(V,E)$ satisfies the following statements with integer coefficients:%
   \COMMENT{}
\begin{enumerate}[\rm(i)]
   \item $H_1 \simeq \Z^{m-d}$, where $m:= |E|$ and $d$ is the rank of~$\Ima\partial_1 ;$\label{H1dim}%
   \COMMENT{}
   \item $C_1/\Ker\partial_1$ is isomorphic to the annihilator of~$\Ker\partial_1$ in~$C^1.$\label{annih1}
   \item The annihilator of~$\Ker\partial_1$ in~$C^1$ equals~$\Ima\delta^0$ if and only if $\Ima\partial_1$ is a direct summand of~$C_0$.\label{annih2}
\end{enumerate}
\end{LEM}

\proof
   By Lemma~\ref{Lemma7} we have $C_1 = D\oplus \Ker\partial_1$ for some submodule~$D$ of~$C_1$.

\eqref{H1dim} Being submodules of the free $\Z$-module~$C_1$, both $D$ and~$\Ker\partial_1$ are free and thus have ranks.%
   \COMMENT{}
   As $C_1 = D\oplus \Ker\partial_1$, these add up to the rank of~$C_1$, which is~$m$. As $D\simeq C_1/\Ker\partial_1\simeq \Ima\partial_1$, the rank of~$D$ is~$d$, and the result follows.%
   \COMMENT{}

\eqref{annih1} Pick a basis $B\cup C$ of~$C_1$ such that $B$ is a basis of~$D$ and $C$ is one of~$\Ker\partial_1$. An element of~$C^1$ sends $\Ker\partial_1$ to zero if and only if it sends $C$ to zero, regardless of where it maps~$B$. Let $\gammaone'$ denote the isomorphism $C_1\to C^1$ which extends linearly from $x\mapsto\gammaone'(x)\in C^1$ with $\gammaone'(x)(y) = \delta_{xy}$ for all $x,y\in B\cup C$.%
   \COMMENT{}
   The annihilator of~$\Ker\partial_1$ then has $\gammaone'(B)$ as a basis,%
   \COMMENT{}
   so it is isomorphic to~$D\simeq C_1/\Ker\partial_1$.%
   \COMMENT{}

\eqref{annih2} Suppose first that $\Ima\partial_1$ is a direct summand of~$C_0$. The annihilator $A$ of~$\Ker\partial_1$ in~$C^1$ includes~$\Ima\delta^0$ by definition of~$\delta$. Its basis~$\gamma'(B)$ from the proof of~\eqref{annih1} maps bijectively%
   \COMMENT{}
   to the basis $\partialone(B)$ of~$\Ima\partial_1$, by the isomorphism theorem for~$\partialone$.%
   \COMMENT{}
   Given~$\psi\in A$, define $\varphi'\colon\Ima\partial_1\to\Z$ via this bijection, i.e.\ by sending $\partialone b$ to $\psi(b)$ for every $b\in B$. By Lemma~\ref{Lalpha}, $\varphi'$~extends to some $\varphi\in C^0$. Then $\psi$ agrees on~$B\cup C$,%
   \COMMENT{}
   and hence on~$C_1$, with $\varphi\circ\partial_1 = \delta(\varphi)$. Thus, $\psi = \delta(\varphi)$, showing $\psi\in\Ima\delta^0$ as desired.

Suppose now that $\Ima\partial_1$ is not a direct summand of~$C_0$. By Lemma~\ref{Lalpha} there exists a homomorphism $\varphi'\colon\Ima\partial_1\to\Z$ that does not extend to any $\varphi\in C^0$. Then $\psi := \varphi'\circ\partialone$ is a 1-cochain in~$A\sm\Ima\delta^0$: since $\partialone(C_1)$ is the entire domain of~$\varphi'$, any $\varphi\in C^0$ satisfying $\psi = \delta(\varphi)$, and hence $\psi = \varphi\circ\partialone$ by definition of~$\delta$, would be an extension of~$\varphi'$.%
   \COMMENT{}
   \endproof

Note that the isomorphism in Lemma~\ref{lem:H1}\,\eqref{annih1} between $C_1/\Ker\partial_1$ and the annihilator of~$\C = \Ker\partial_1$ need not be canonical,%
   \COMMENT{}
   in that we may not be able to choose the isomorphism~$\gammaone'$ used in the proof as~$\gammaone$. At the end of this paper, after Theorem~\ref{cor:sp.tree summand}, we shall see an example where $\Ima\delta^0$ is the entire annihilator~$A$%
   \COMMENT{}
   of~$\Ker\partial_1$ in~$C^1$, so that $\B := \gammaone^{-1}(\Ima\delta^0)$ satisfies $\B = \gammaone^{-1}(A) = \C^\bot$,%
   \COMMENT{}
   but where $\B\oplus\C$ is a proper submodule of~$C_1$ while $\B' := (\gamma')^{-1}(A) = \big((\gamma')^{-1}\circ\gammaone\big) (\B)$%
   \COMMENT{}
   satisfies $\B'\oplus\C = D\oplus\Ker\partial_1 = C_1$.

Another interesting aspect of Lemma~\ref{lem:H1} is that if $\Ima\partial_1$ is not a direct summand of~$C_0$~-- and we shall see a simple example of this in a moment~-- then $C_1/\Ker\partial_1$ is isomorphic both to the annihilator of~$\Ker\partial_1$,%
   \COMMENT{}
   which contains~$\Ima\delta^0$ properly,%
   \COMMENT{}
   and to the submodule~$\D$ of~$\Ima\delta^0$ as in Lemma~\ref{D}. Thus, the annihilator of~$\Ker\partial_1$ is isomorphic to the proper submodule~$\D$ of itself.%
  \COMMENT{}
   Since the annihilator of~$\Ker\partial_1$, denoted as~$A = \gammaone'(D)$ in the proof of Lemma~\ref{lem:H1}, is free,%
   \COMMENT{}
   $\D$~is not a direct summand of it, by Lemmas \ref{ETS} and~\ref{Lalpha}.%
   \COMMENT{}

\looseproof{Proof of Theorem~\ref{Thm5}}
   The equivalence of our assertions \eqref{A=Im} and~\eqref{Imd} is Lemma~\ref{lem:H1}\,\eqref{annih2}. Assertion~\eqref{B=Cbot} is just the translation of~\eqref{A=Im} via~$\gamma_1$.%
   \COMMENT{}

Let us show that~\eqref{A=Im} is equivalent to~\eqref{HomH1}. 
   \COMMENT{}
   $\!\!$The homomorphism ${\psi\mapsto\psi\restricts\Ker\partial_1}$ from $C^1$ to~${\rm Hom}(\Ker\partial_1,\Z) = {\rm Hom}(H_1,\Z)$ is surjective, because every homomorphism $\Ker\partial_1\!\to\Z$ extends to one in~$C^1$ by Lemmas \ref{Lemma7} and~\ref{Lalpha}. By the isomorphism theorem it thus defines an isomorphism from $C^1/A$ to~${\rm Hom}(H_1,\Z)$, where $A$ is the annihilator of~$\Ker\partial_1$ in~$C^1$ and hence its kernel. If $A = \Ima\delta^0$, which is assertion~\eqref{A=Im}, we thus have an isomorphism from $H^1 = C^1/\Ima\delta^0$ to ${\rm Hom}(H_1,\Z)$ as claimed in~\eqref{HomH1}. If \eqref{A=Im} fails, then $\Ima\delta^0$ is properly contained in~$A$. Our map $\psi\mapsto \psi\restricts\Ker\partial_1$ is still well defined on the classes in~$H^1\!$,%
   \COMMENT{}
   and thus defines an epimorphism from $H^1$ to~${\rm Hom}(H_1,\Z)$. But this is not injective, so \eqref{HomH1} fails too.

Let us now prove~\eqref{H1=H1}, assuming~\eqref{Imd}. In order to show that $H_1$ and~$H^1$ are canonically isomorphic, consider again the dual sequences shown in Figure~\ref{LAexact}. While the top sequence is still exact, its dual sequence below it need not be exact,\penalty-200\ prima facie, now that we have integer coefficients. However we shall prove that in fact it is,%
   \COMMENT{}
   so that
 $$H_1 = \Ker\partial_1\simeq (\Ker\partial_1)^* = \Ima i^*\simeq C^1/\Ker i^* = C^1/\Ima\delta^0 = H^1$$
 with the required isomorphism ${\pi\colon C^1\!/\Ker i^*\!\to\Ima i^*\!}$.%
   \COMMENT{}
   The only nontrivial assertions here%
   \COMMENT{}
   are the second and the penultimate `=', so let us address these in turn.

$\!$For the second~`=' we have to show that every homomorphism ${\psi'\!\colon\!\Ker\partial_1\to\Z}$ has the form~$i^*(\psi) = \psi\circ i$ for some~$\psi\in C^1$, i.e, extends to a homomomorphism~$\psi$ defined on all of~$C_1$. This holds by Lemmas~\ref{Lalpha} and~\ref{Lemma7}.%
   \COMMENT{}

For the penultimate~`=' we have to show that $\Ima\delta^0 = \Ker i^*$. So let us again examine the homomorphism~$i^*\colon C^1\to (\Ker\partial_1)^*$. It sends a given $\psi\in C^1$ to the homomorphism $i^*(\psi)\colon \Ker\partial_1 \to\Z$ which maps $x\in\Ker\partial_1$ to $\psi(i(x)) = \psi(x)$.%
   \COMMENT{}
   Thus, $i^*(\psi) = \psi\restricts\Ker\partial_1$, and so $\Ker i^*$ consists of those $\psi\in C^1$ that send~$\Ker\partial_1$ to~$0\in\Z$. Clearly all $\psi\in\Ima\delta^0$ do that, because $\delta^0$ is defined as the dual of~$\partial_1$. Thus trivially $\Ima\delta^0\sub\Ker i^*$, and it remains to show the converse inclusion.

We have to show that every homomorphism $\psi\in C^1$ that sends $\Ker\partial_1$ to zero has the form $\psi = \delta^0(\varphi)$ for some $\varphi\in C^0$. In order to find such~$\varphi$ note that, since $\psi(\Ker\partial_1) = 0$, these~$\psi$ are well defined on~$C_1/\Ker\partial_1$; let $\tilde\psi\colon {C_1/\Ker\partial_1\to\Z}$ be the induced map, which sends $[x]\in C_1/\Ker\partial_1$ to~$\psi(x)\in\Z$. Define $\varphi'\colon{\Ima\partial_1\to\Z}$ as $\varphi':=\tilde\psi\circ\pi^{-1}$, where $\pi$ is the isomorphism $C_1/\Ker\partial_1\to\Ima\partial_1$ that sends $[x]$ to~$\partial_1 x$. By Lemma~\ref{Lalpha} and our assumption~\eqref{Imd} that $\Ima\partial_1$ is a direct summand of~$C_0$, our~$\varphi'$ extends to a homomorphism~$\varphi\in C^0$. Then $\psi = \delta^0(\varphi)$ as desired, since for all $x\in C_1$ we have $\psi(x) = \tilde\psi([x]) = \varphi'(\partial_1 x) = \varphi(\partial_1 x)$.

To complete our proof we show that \eqref{H1=H1} implies~\eqref{A=Im}. As we noted earlier, the annihilator of~$\Ker\partial_1$ in~$C^1$ is~$\Ker i^*$. If \eqref{A=Im} fails, then this is not~$\Ima\delta^0$. Then $C^1/\Ker i^*\ne C^1/\Ima\delta^0$, so $H_1$ and~$H^1$ are not canonically isomorphic.
  \endproof

It can happen that $H_1$ and~$H^1$ are not isomorphic at all, not just not canonically. Let us show this now.

\goodbreak

\begin{EX}\label{MainExample}
	The hypergraph $(V,E)$ with $E = \{e_1,e_2,e_3\}$ and ${V \!= \{v_1,v_2,v_3\}}$ where $e_i = (\{v_j,v_k\},\{v_i\})$ whenever $\{i,j,k\} = \{1,2,3\}$ satisfies the following:
   \begin{enumerate}[\rm (i)]\itemsep=0pt
     \item The homomorphism $\partial_1$ is injective.\label{partialinj}
     \item The image of $\partial_1$ is not a direct summand of~$C_0$.\label{Impartial}
     \item The homomorphism $\delta^0$ is injective.\label{deltainj}
     \item The image of $\delta^0$ is not a direct summand of~$C^1$.\label{Imdelta}
     \item The annihilator of $\Ker\partial_1$ in~$C^1$ contains $\Ima\delta^0$ properly.\label{IminA}
     \item The module $\D = \Ima (\deltanought\circ\gamma\circ\partial_1)$ is properly contained in $\Ima\delta^0$.\label{DinIm}
     \item $H_1\not\simeq H^1$.\label{H1H1}
   \end{enumerate}
 Moreover, this hypergraph has no algebraic spanning tree over the integers.%
   \COMMENT{}
\end{EX}

\proof (Sketch)
   (\ref{partialinj}) Direct inspection shows that $\partial E = (\partial e_1, \partial e_2, \partial e_3)$ is linearly independent%
   \COMMENT{}
   and $\Ker\partial_1 = \{0\}$.%
   \COMMENT{}

(\ref{Impartial}) As both $C_0$ and~$C_1$ have rank~3,%
   \COMMENT{}
   assertion~\eqref{partialinj} implies by Lemma~\ref{ETS} that either $\Ima\partial_1 = C_0$ or $\Ima\partial_1$ has non-unit elementary divisors as a submodule of~$C_0$.%
   \COMMENT{}
   But $\Ima\partial_1\ne C_0$: for example, it is easy to show that $\partial E$ does not generate the singleton chains~$v_i\in C_0$.%
   \COMMENT{}
   Hence $\Ima\partial_1$ has non-unit elementary divisors, which implies~(\ref{Impartial}) by Lemma~\ref{Lalpha}.

(\ref{deltainj}) This is again easy to check by direct inspection.%
   \COMMENT{}

\eqref{Imdelta} The proof of this is analogous to the proof of~\eqref{Impartial}; use assertion~\eqref{IminA}, to be proved independently below, for the required fact that $\Ima\delta^0\ne C^1$.%
   \COMMENT{}

(\ref{IminA}) follows from~(\ref{Impartial}) by Lemma~\ref{lem:H1}\,\eqref{annih2}.

(\ref{DinIm}) By \eqref{Impartial} we have $\Ima\partialone\subsetneq C_0$, so there exists $\varphi\in C^0\sm \gamma(\Ima\partialone)$. By~\eqref{deltainj}, $\deltanought(\varphi)\in\Ima\delta^0\sm\D$.

(\ref{H1H1}) We have $H_1 = \Ker\partial_1 = \{0\}$ by~\eqref{partialinj}%
   \COMMENT{}
   but $H^1 = C^1/\Ima\delta^0 \ne \{0\}$ by~(\ref{IminA}).%
   \COMMENT{}

\goodbreak

Let us finally show that our hypergraph has no spanning tree over the integers. As $\C = \Ker\partial_1 = \{0\}$ must, for any spanning tree $T\sub E$, have a basis indexed by~$E\sm T$, we would have $E\sm T = \es$ and thus $T=E$. Since, by the first spanning tree axiom, the homomorphisms~$\gamma(x_t)$ for the associated basis $(\,x_t\mid t\in T\,)$ of~$\B$ send $t'\in T$ to~$\delta_{tt'}$, the fact that $T=E$ implies for all $t\in T$ that $\gamma(x_t)(e) = \delta_{te} = \gamma(t)(e)$ for all $e\in E$, so $\gamma(x_t) = \gamma(t)$ and hence $x_t = t$ for all $t\in T=E$. So $(\,x_t\mid t\in T\,) = (e_1, e_2, e_3)$ must be a basis of~$\B = \gamma^{-1} (\Ima\delta^0)$, which makes $\gamma(E)$ a basis of~$\Ima\delta^0$. But it is not: $\gamma(E)$~is a basis of~$C^1$, which contains~$\Ima\delta^0$ properly by~(\ref{IminA}).%
   \COMMENT{}
   \endproof

\section{Direct summands over the integers}\label{sec:DirectSummands}

In view of Theorem~\ref{Thm5}, the central remaining problem from Section~\ref{sec:H1} is to investigate which hypergraphs are algebraically graph-like. It would be interesting to see any structural conditions on hypergraphs that imply graph-likeness.

This is part~(i) of the following problem. Is its part~(ii) related to~(i)?

\begin{PROB}\label{images}
Which hypergraphs over the integers satisfy
 \begin{enumerate}[\rm(i)]\itemsep=0pt\vskip-\smallskipamount\vskip0pt
 	\item $\Ima\partial_1$ is a direct summand of~$C_0\,?$\label{imagespartial}%
      \COMMENT{}
    \item $\Ima\delta^0$ is a direct summand of~$C^1\,?$\label{imagesdelta}%
    \COMMENT{}    
 \end{enumerate}
\end{PROB}

\goodbreak

One structural condition implying this is that our hypergraph is  a graph:%
   \COMMENT{}

\begin{THM}\label{Thm11}
	If $(V,E)$ is a graph then $\Ima\partial_1$\,and $\Ima\delta^0\!$, taken over the integers, are direct summands of~$C_0$ and~$C^1\!$, respectively.
\end{THM}

\proof
   Let $T\sub E$ be (the oriented edge set of) a~spanning tree, with root~$r\in V\!$ say. We first prove that $\Ima\partial_1$ is a direct summand of~$C_0$, by showing that $\partialone T = (\,\partial t\mid t\in T\,)$ is a basis of~$\Ima\partial_1$ and the extended family $(r|\partial T)$ is a basis of~$C_0$.

Since $[t]\mapsto\partial t$ is a (well-defined) isomomorphism from $C_1/\Ker\partial_1$ to~$\Ima\partial_1$, showing that $\partialone T$ is a basis of~$\Ima\partial_1$ is equivalent to showing that $(\,[t]\mid t\in T\,)$ is a basis of~$C_1/\Ker\partial_1$. Let us do this first.

For a proof that $(\,[t]\mid t\in T\,)$ is linearly independent, note first that $\sum_{t\in T} n_t [t] = 0\in C_1/\Ker\partial_1$ if and only if $x:= \sum_{t\in T} n_t t\in\Ker\partial_1$. So let us show that this holds only when $n_t = 0$ for all~$t\in T$. Suppose the contrary, that
 $$T':= \{\,t\in T\mid n_t\ne 0\,\}\ne\es\,.$$
 Let $v$ be a leaf of (the forest induced by the edges~in)~$T'$. Then $v$ is incident with exactly one edge~$t'$ in~$T'$, so
 $$\textstyle\partialone x = \partialone\sum_{t\in T'} n_t t = \sum_{t\in T'} n_t \partialone t\in C_0$$
 has coefficient $\pm n_{t'}\ne 0$%
   \COMMENT{}
   at~$v$. Thus, $\partialone x\ne 0\in C_0$ and hence $x\notin \Ker\partial_1$, as desired.

Next, let us show that $(\,[t]\mid t\in T\,)$ generates~$C_1/\Ker\partial_1$. Let ${[x]\in C_1/\Ker\partial_1}$ be given, with $x = \sum_{e\in E} n_e e$ say. Let 
 $$\textstyle x_0 := \sum_{e\in E\sm T} n_e x_e\,,$$%
   \COMMENT{}
   where $x_e\in \Ker\partial_1$ is the (oriented) fundamental cycle of~$e$ with respect to~$T$. As $x_e(e') = \delta_{ee'}$ for all $e,e'\in E\sm T$, our 1-chains $x$ and~$x_0$ have the same coefficients~$n_e$ on~$E\sm T$. These coefficients vanish for~$x-x_0$, so 
 $$\textstyle x_1 := x-x_0 = \sum_{t\in T} m_t t$$
 for suitable coefficients~$m_t$. As $x-x_1 = x_0\in\Ker\partial_1$, the chains $x$ and~$x_1$ re\-pre\-sent the same class in~$C_1/\Ker\partial_1$. Thus, $[x] = [x_1] = \sum_{t\in T} m_t [t]$ as desired.

We have shown that $\partialone T$ is a basis of~$\Ima\partial_1$. To show that $(r|\partial T)$ generates~$C_0$, let $y = \sum_{v\in V} n_v v\in C_0$ be given. Choose coefficients~$\ell_t$ for all $t\in T$ inductively, as follows. Let $T = T_0,\dots,T_k = \{r\}$ be such that $T_{i+1}$ is obtained from~$T_i$ by deleting all its (edges incident with) leaves: vertices of degree~1 other than the root~$r$. For $i=0,\dots,k-1$ in turn we can now choose coefficients $\ell_t$ for $t\in T_i\sm T_{i+1}$ so that $\sum_{t\in T} \ell_t\partial t$ has the same coefficients as~$y$ at all $v$ except possibly at~$r$.%
   \COMMENT{}
   This can be adjusted by adding $\ell_r r$ for some suitable~$\ell_r$, so that
 $$\textstyle \ell_r r + \sum_{t\in T} \ell_t\partial t = y$$
 as desired.

We have shown that $(r|\partial T)$ generates~$C_0$; let us show that it is linearly independent.%
   \footnote{This also follows by a known algebraic property of free $\Z$-modules of finite rank, which is that they cannot have generating sets smaller than their rank. As $(r|\partial T)$ has the same size as~$V\!$, which is a basis of~$C_0$, it therefore cannot contain a smaller generating set, which it would if it was not linearly independent.}
   Following their inductive definition, it is easy to see that the coefficients~$\ell_t$ and~$\ell_r$ of the generators are unique, given~$y\in C_0$. In particular, the only way to generate $0\in C_0$ is by choosing $\ell_r = \ell_t = 0$ for all~$t\in T$. Thus, $(r|\partial T)$ is linearly independent. This completes our proof that $\Ima\partial_1$ is a direct summand of~$C_0$.

\goodbreak

Let us now show that $\Ima\delta^0$ is a direct summand of~$C^1$. Every $t\in T$ separates the tree $T$ into two components $T_t^-$ and~$T_t^+$ of $T-t$, with $t$ oriented from~$T_t^-$ to~$T_t^+$. Let $\varphi_t\in C^0$ map $V(T_t^-)$ to~$0$ and $V(T_t^+)$ to~$1$, and put $\beta_t := \delta^0(\varphi_t)$. (This is the 1-cochain corresponding%
   \COMMENT{}
   to the fundamental cut of~$t$ in~$C^1$.) Our first aim is to  prove that these~$\beta_t$ generate~$\Ima\partial^0$.

To show this, let~$\psi = \delta^0(\varphi)$ be given. Let $\psi':=\sum_{t\in T} n_t\beta_t$, where $n_t = \varphi(v) - \varphi(u)$ for $t=uv$.%
   \COMMENT{}
   Note that $\psi'$ agrees with~$\psi$ on~$T$, because $\beta_t(t') = \delta_{tt'}$ and $n_t = \psi(t)$ for each~$t$; let us show that they agree on $E\sm T$ too.

So let any chord $e = uv\in E\sm T$ be given. Let $v_0,\dots,v_k\in V\!$ be the vertices on the path in~$T$ from $u=v_0$ to~$v=v_k$, and write $t_i$ for the edge~$v_{i-1} v_i$ on that path, $i=1,\dots,k$. Then
 $$\textstyle\psi(e) = \varphi(v) - \varphi(u) = \sum_{i=1}^k(\varphi(v_i) - \varphi(v_{i-1})) = \sum_{i=1}^k n_{t_i} \beta_{t_i} (t_i) = \psi'(e);$$%
   \COMMENT{}
 for the last equality note that $\beta_t(e)=1$ for $t = t_1,\dots,t_k$ but $\beta_t(e) = 0$ for all other $t\in T$.%
   \footnote{This is the oriented version of the well-known property of unoriented graphs that a chord lies in the fundamental cuts of precisely those tree edges that lie on its fundamental cycle.}
   This completes our proof that $(\,\beta_t\mid t\in T\,)$ generates~$\Ima\delta^0$.

Next, let us show that our $\beta_t$ for $t\in T$ and the maps $\psi_e = \gammaone(e)$ for~$e\in E\sm T$%
   \COMMENT{}
   together generate~$C^1$. Let $\psi = C^1$ be given. Since $(\,\psi_e\mid e\in E\,)$ is a basis of~$C^1$ (the dual of the basis~$E$ of~$C_1$), we can express~$\psi$ as
 $$\textstyle\psi = \sum_{e\in E} n_e \psi_e$$
 with suitable coefficients~$n_e$. Let us use some of these to define
 $$\textstyle\psi' := \sum_{t\in T} n_t \beta_t\,.$$
 Once more, $\psi'$ and~$\psi$ agree on~$T$, because $\beta_{t'}(t) = \delta_{tt'}$ and $\psi_e(t) = \delta_{et}$. On~$E\sm T$ our two maps need not agree; let $\ell_e := \psi(e) - \psi'(e)$ for chords $e\in E\sm T$. Then 
 $$\textstyle\psi'' := \psi - \psi' = \sum_{e\in E\sm T} \ell_e\psi_e\,,$$
 \vskip-\smallskipamount\noindent
 so
 $$\textstyle\psi = \psi'+\psi'' = \sum_{t\in T} n_t\beta_t + \sum_{e\in E\sm T}\ell_e\psi_e$$
 as desired.%
   \COMMENT{}
   
We have shown that the $\beta_t$ for $t\in T$ and the $\psi_e$ for~$e\in E\sm T$ together generate~$C^1$. In fact, the coefficients $n_t$ and~$\ell_e$ above are easily seen to be unique: the~$n_t$ are, because $\beta_{t'}(t) = 0 = \psi_e(t)$ for all $t'\ne t$ and all~$e$,%
   \COMMENT{}
   and once the $n_t$ are determined so are the~$\ell_e$.%
   \COMMENT{}
   In particular, the only way to generate $0\in C^1$ is with $n_t = 0 = \ell_e$ for all $t$ and~$e$, so our $\beta_t$ and~$\psi_e$ in fact form a basis of~$C^1$.%
   \footnote{This also follows, once more, from the fact that finitely generated free $\Z$-modules have a well-defined rank.%
   \COMMENT{}
   But this is a nontrivial property that seems unnecessary to invoke here.}%
   \COMMENT{}

Being part of this basis, $(\,\beta_t\mid t\in T\,)$ is linearly independent. We showed earlier that it generates~$\Ima\delta^0$, so it is a basis of~$\Ima\delta^0$. As this basis is part of our basis $(\,\beta_t\mid t\in T\,)\cup (\,\psi_e\mid e\in E\sm T\,)$ of~$C^1$, we have established our claim that $\Ima\delta^0$ is a direct summand of~$C^1$.
   \endproof

The reader will have noticed that our proof of the second assertion of Theorem~\ref{Thm11}, that $\Ima\delta^0$ is a direct summand of~$C^1$, used graphs only in as much they satisfy our first spanning tree axiom from Section~\ref{sec:SpTrees}: that the cochains~$\beta_t$ corresponding to the fundamental cuts of~$T$ form a basis of~$\Ima\delta^0$ satisfying ${\beta_t(t') = \delta_{tt'}}$. (This contrasts with the first part of our proof, in which our extension of $\partial T$ to a basis of~$C_0$ by `adding a vertex' does seem to use graphs in an essential way.)%
   \COMMENT{}

We thus have, as a corollary of the proof of Theorem~\ref{Thm11}, a positive solution to Problem~\ref{images}\,\eqref{imagesdelta} for hypergraphs with spanning trees:

\begin{THM}\label{cor:sp.tree summand}
For every hypergraph over the integers that has an algebraic spanning tree, $\Ima\delta^0$ is a direct summand of~$C^1$.\noproof
\end{THM}

In view of Theorem~\ref{Thm11} it is instructive to revisit the example, mentioned briefly before Problem~\ref{BisCbot}, of the two-vertex graph with two parallel edges, $e$~and~$t$ say. Let us once more use the notation of $\C = \Ker\partial_1$ and~$\B = \gamma^{-1}(\Ima\delta^0)$. By Proposition~\ref{1.9.5} we have both $\C^\bot = \B$  and~$\B^\bot = \C$. But it is easy to see that $\C$ and~$\B$ do not, together, generate all of~$\E = C_1$; for example, they do not generate the edge~$t$.%
   \COMMENT{}

On the other hand, it is also easy to show that $\C\cap\B = \{0\}$.%
   \COMMENT{}
   This implies that $\C$ and~$\B$ form a direct sum: every element of $\C+\B$ has a unique representation as $c+b$ with $c\in\C$ and $b\in\B$.%
   \COMMENT{}
   But this direct sum $\C\oplus\B$ is properly contained in~$\E$.

Nonetheless, $\C$~is a direct summand of~$\E$ by Lemma~\ref{Lemma7},
and $\B$ is a direct summand of~$\E$ by Theorem~\ref{Thm11}. However the direct complement of~$\C$ in~$\E$ is not~$\B$, and the direct complement of $\B$ in~$\E$ is not~$\C$.

All this plays out explicitly as follows. We have $\C = \{\,n(e-t)\mid n\in\Z\,\}$, which is a direct summand of~$\E$ since $(e-t,t)$ is a basis of~$\E$. Similarly we have $\Ima\delta^0 = \{\,n(\psi_e + \psi_t)\mid n\in\Z\,\}$,%
   \COMMENT{}
   which is a direct summand of~$C^1$ since $(\psi_e+\psi_t,\psi_t)$ is a basis of~$C^1$.%
   \COMMENT{}
   The corresponding%
   \COMMENT{}
   direct summand of~$\E=C_1$ is $\B = \{\,n(e+t)\mid n\in\Z\,\}$,%
   \COMMENT{}
   the span of the first element of the basis $(e+t,t)$ of~$\E$.

We have thus shown that $\C\oplus t\Z = \E = \B\oplus t\Z$. So $t\Z$ is a direct summand of~$\E$ that individually `complements' each of $\C$ and~$\B$, which are orthogonal complements of each other in $\C\oplus\B\subsetneq\E$.%
   \COMMENT{}%
   \COMMENT{}

This little example also throws a light on the proof of Theorem~\ref{Thm11}, by showing why we had to use such a mix of generators rather than, for example, just fundamental cycles and cuts. Indeed, we could not have used $[x_t]$ instead of~$[t]$ in the first part (where $x_t = \gamma^{-1}(\beta_t)$ is the fundamental cut of~$t$ with respect to~$T$), and we could not have used $\gamma(x_e)$ instead of~$\psi_e$ in the second part (where $x_e$ is the fundamental cycle of~$e$ with respect to~$T$), because $\C$ and~$\B$ together~-- let alone the fundamental cycles and cuts together~-- do not generate all of~$\E$.

\section*{Acknowledgement}

I would like to thank Ebrahim Ghorbani and Ruben Melcher for independently pointing out the forward direction of Lemma~\ref{sp.tree bases}, which was not included in earlier drafts of this paper. The proof now included is Melcher's. Ghorbani also pointed out to me how some of the proofs could be rewritten in terms of matrices and their Smith normal form, as indicated at the end of the introduction.

\newpage

\bibliographystyle{plain}
\bibliography{collective}

\end{document}